\documentclass[11pt,reqno]{amsart}

\usepackage{amsmath,amsthm,amssymb,amsfonts,graphicx}
\newtheorem{theorem}{Theorem}[section]

\newtheorem{definition}[theorem]{Definition}
\newtheorem{remark}[theorem]{Remark}

\newtheorem{claim}{Claim}

\allowdisplaybreaks

\numberwithin{equation}{section}

\font\script=rsfs10 at 11pt
\def\H{{\mbox{\script H}\;}}
\def\N{\mathbb N}
\def\R{\mathbb R}
\def\S{\mathbb S}
\def\eps{\varepsilon}
\def\eu{{\rm eucl}}
\def\step#1#2{\par\noindent{\underline{\it Step~#1.}}\emph{ #2}\\}
\def\case#1#2{\par\noindent{\underline{\it Case~#1.}}\emph{ #2}\\}
\def\part#1#2{\par\noindent{\underline{\it Part~#1.}}\emph{ #2}\\}
\def\XXint#1#2#3{{\setbox0=\hbox{$#1{#2#3}{\int}$} \vcenter{\vspace{-1pt}\hbox{$#2#3$}}\kern-.5\wd0}}
\def\Xint#1{\mathchoice {\XXint\displaystyle\textstyle{#1}}{\XXint\textstyle\scriptstyle{#1}}{\XXint\scriptstyle\scriptscriptstyle{#1}}{\XXint\scriptscriptstyle\scriptscriptstyle{#1}}\!\int}
\def\intmed{\Xint{\hbox{---}}}

\def\res{\mathop{\hbox{\vrule height 7pt width .5pt depth 0pt \vrule height .5pt width 6pt depth 0pt}}\nolimits}
\newcounter{mt}
\def\maintheorem#1#2#3{\par \medskip \noindent {\bf Theorem~\mref{#1}}~(#2).~{\it #3}\par}
\def\mref#1{\Alph{#1}}
\def\maintheoremdeclaration#1{\stepcounter{mt}\newcounter{#1}\setcounter{#1}{\arabic{mt}}}
\maintheoremdeclaration{boundedness}
\maintheoremdeclaration{eps-epsbeta}
\title[$\eps-\eps^\beta$ property, boundedness of isoperimetric sets and applications]{The $\eps-\eps^\beta$ property, the boundedness of isoperimetric sets in $\R^N$ with density, and some applications}

\author{E. Cinti}
\author{A. Pratelli}

\begin{document}

\begin{abstract}
We show that every isoperimetric set in $\R^N$ with density is bounded if the density is continuous and bounded by above and below. This improves the previously known boundedness results, which basically needed a Lipschitz assumption; on the other hand, the present assumption is sharp, as we show with an explicit example. To obtain our result, we observe that the main tool which is often used, namely a classical ``$\eps-\eps$'' property already discussed by Allard, Almgren and Bombieri, admits a weaker counterpart which is still sufficient for the boundedness, namely, an ``{$\eps-\eps^\beta$}'' version of the property. And in turn, while for the validity of the first property the Lipschitz assumption is essential, for the latter the sole continuity is enough. We conclude by deriving some consequences of our result about the existence and regularity of isoperimetric sets.
\end{abstract}

\maketitle

\section{Introduction}

This paper deals with the isoperimetric problem in $\R^N$ with density. More precisely, we consider a given l.s.c. function $f:\R^N\to \R^+$ (the ``density'') and we define the weighted volume and perimeter of a set $E\subseteq\R^N$ as
\begin{align}\label{defweight}
| E|_f := \int_E f(x)\, d\H^N(x)\,, &&
P_f(E):= \int_{\partial^* E} f(y)\, d\H^{N-1}(y)\,,
\end{align}
where for every set $E$ of locally finite perimeter we denote as usual by $\partial^* E$ its reduced boundary, while $P_f(E)=+\infty$ for every set which is not locally of finite perimeter (the basic properties of sets of finite perimeter will be briefly recalled in Section~\ref{sec:finper}). The \emph{isoperimetric problem}, then, consists in searching for sets of minimal (weighted) perimeter among those of fixed (weighted) volume. The study of the isoperimetric problem in $\R^N$ with density has been deeply studied in last years, also because of its close connection with the isoperimetric problems on Riemannian manifolds (a short and incomplete list is~\cite{CMV,CFMP,DHHT,FM,FMP2,Morgan2005,MP,RCBM}).\par

The three main questions which one wants to understand are usually the existence, the boundedness, and the regularity of isoperimetric sets (only in very specific examples one can try to determine explicitely the minimizers). Concerning the boundedness, it is to be pointed out that it is not only interesting by itself, but it is also important when proving the existence (roughly speaking, when trying to show the existence of an isoperimetric set for volume $m$, it is useful to know that there do not exist unbounded isoperimetric sets for volumes less than $m$).\par

We will be able to give some new results concerning all the three questions; in particular, we will find a sharp boundedness theorem (Theorem~\ref{obvious}).\par

One important tool in many of the works on isoperimetric problems is a classical ``$\eps-\eps$'' property already discussed by Allard, by Almgren, and by Bombieri since the 1970's (see for instance~\cite{All1, All2, Alm, BigReg, Bom}); this property basically means that a certain set can be locally modified in order to increase its volume by $\eps$, while the perimeter increases at most by $C\eps$ (we leave the formal definitions to Section~\ref{preliminaries}). This property is trivial to establish if the density and the set are supposed to be regular enough, but its validity is also known if the set is even just of locally finite perimeter and the density is only Lipschitz continuous (see~\cite{Morgan2003,Morganbook}); on the other hand, it is very easy to observe that the validity may fail as soon as the density is not Lipschitz. This is more or less the reason why most of the different boundedness and regularity results in this context use at least a Lipschitz assumption on the density.\par

In this paper, we start from the following observation. It is classical and very easy to prove that, if an isoperimetric set fulfills the $\eps-\eps$ property, then it must be bounded; but in fact, to get its boundedness, a weaker property is really needed, namely, the ``$\eps-\eps^\beta$'' property, which states that it is possible to increase its volume by $\eps$ and its perimeter at most by $C\eps^\beta$, for some $\beta\geq\frac{N-1}N$ (this is the content of our Theorem~\mref{boundedness}). Then, our main result implies that the $\eps-\eps^{\frac{N-1}N}$ property is always true for any set of locally finite perimeter whenever the density is continuous. Moreover, for every $0< \alpha\leq 1$ there exists some $\beta=\beta(\alpha,N)$, with $\frac{N-1}N<\beta\leq 1$ such that, if the density is $C^{0,\alpha}$, then the $\eps-\eps^\beta$ property holds (this will be proved in Theorem~\mref{eps-epsbeta}). Putting together the two results, the following consequence will be immediate (the meaning of ``essentially bounded'' density is clarified in Definition~\ref{essbdd}, anyway any density which is bounded from above and below is essentially bounded).
\begin{theorem}\label{obvious}
Assume that $f$ is continuous and essentially bounded. Then every isoperimetric set is bounded.
\end{theorem}
We underline that this result is sharp. In fact, many examples (see for instance those in~\cite{MP}) show the existence of unbounded isoperimetric sets for densities which are unbounded from above or from below; and on the other hand, in Section~\ref{example} we are able to build the example of an unbounded isoperimetric set for a density which is bounded both from above and from below but not continuous. As we said above, Theorem~\ref{obvious} is stronger than the previously known results, which all needed at least a Lipschitz assumption. More precisely, to give a comparison, we just recall the following very recent boundedness result.
\begin{theorem}[\cite{MP}, Corollary~5.11]
Let $E$ be an isoperimetric set in $\R^N$ with a ${\rm C}^1$ density $f$. Then E is bounded if any of the following three hypotheses hold:
\begin{enumerate}
\item $N=2$ and $f$ is increasing, or
\item $f$ is radial and increasing, or
\item $|Df|\leq Cf$.
\end{enumerate}
\end{theorem}

Let us now briefly pass to describe our contributions to the questions of the existence and regularity of isoperimetric sets, which come as applications of Theorems~\mref{boundedness} and~\mref{eps-epsbeta}. Concerning the existence, we will only recall that some existence results available in the literature require, as an \emph{a priori} assumption, the boundedness of the isoperimetric sets. Since all these results concern densities which are essentially bounded and continuous, the boundedness assumption can be removed, because it now directly follows from Theorem~\ref{obvious}.\par

Concerning the regularity, instead, we will recall some very classical results, and we check their consequences in view of the $\eps-\eps^\beta$ property that we have established. The theorem that we obtain, Theorem~\ref{regu}, says that if the density $f$ is essentially $C^{0,\alpha}$ then any isoperimetric set is of class $C^{0,\frac\alpha{2N(1-\alpha)+2\alpha}}$. Stronger regularity results are contained in the forthcoming paper~\cite{CP2}.\par\medskip

Observe that, as usual, results on $\R^N$ with density admit counterparts in the context of Riemannian manifolds; however, we do not study the extension here (for an overview of the known results in this direction, see for instance~\cite{Morgan2003}).\par\bigskip

The plan of the paper is the following. In Section~\ref{preliminaries} we give all the formal definitions and the claims of our main results, while in Section~\ref{sec:finper} we recall some basic properties of the sets of finite perimeter. Then, in Sections~\ref{bounded} and~\ref{holder} we give the proof of Theorems~\mref{boundedness} and~\mref{eps-epsbeta}. Later on, in Section~\ref{example} we give the example of a situation where an isoperimetric set is unbounded, while the density is essentially bounded but not continuous. Finally, in Section~\ref{applications} we discuss the questions of the existence and regularity of isoperimetric sets.

\subsection{Preliminary definitions and claims of the main theorems\label{preliminaries}}

This section is devoted to present the relevant definitions that we will need during the paper, and to claim our main results.

We consider a given l.s.c. function $f:\R^N\to \R^+=[0,+\infty]$, such that the points $x$ for which $f(x)=0$ or $f(x)=+\infty$ are locally finite, and we work with the weighted notion of volume and perimeter given by~(\ref{defweight}). For brevity we will denote $\H^k_f:=f \H^k$ for any $k\in\{0, \dots ,N\}$, so that definition~\eqref{defweight} can be rewritten as $|E|_f=\H^N_f(E)$ and $P_f(E)=\H^{N-1}_f(\partial^*E)$; given an open set $A\subseteq\R^N$, we will denote the relative perimeter of $E$ in $A$ as $P_f(E,A)=\H^{N-1}_f(A \cap \partial^* E)$. Sometimes we will need to consider the Euclidean volume, or perimeter, or relative perimeter, of a set $E$, which will be denoted by $|E|_\eu$, $P_\eu(E)$ or $P_\eu(E,A)$ respectively (while we will not use the notation $P(E)$ or $|E|$ to avoid ambiguity). We will always call $B_R$ the open ball centered at the origin and with radius $R$, and $B_R(x)$ the ball centered at $x$ and with radius $R$. The next definition is sometimes useful.

\begin{definition}
For every set $E\subseteq\R^N$, we define the \emph{spherical rearrangement} $E^*\subseteq\R^N$ as
\[
E^*:=\Big\{ x\in\R^N:\, x_1\geq g\big(|x|\big) \Big\}\,,
\]
where $g:\R^+\to\R^+$ is the unique function such that for every $R>0$ one has
\[
\H^{N-1}\big( E^* \cap \partial B_R\big) = \H^{N-1}\big( E \cap \partial B_R\big)\,.
\]
\end{definition}
The definition of spherical rearrangement is not so much useful for a generic density $f$, but it becomes very important when $f$ is radial, thanks to the following result, whose proof can be found for instance in~\cite{MP} or~\cite{FMP}.
\begin{theorem}\label{sphsym}
Assume that $f$ is radial and $E\subseteq\R^N$. Then, one has
\begin{align*}
\big|E^* \big|_f= \big|E \big|_f\,, && P_f(E^*)\leq P_f(E) \,.
\end{align*}
\end{theorem}

Let us give now the particular definitions that we will use in this paper.
\begin{definition}\label{defepepbe}
Let $E\subseteq \R^N$ be a set of locally finite perimeter, $0\leq \beta\leq 1$, and $C>0$. We say that \emph{$E$ fulfills the $\eps-\eps^\beta$ property with constant $C$} if there exist a ball $B$ and a constant $\bar\eps>0$ such that, for every $|\eps|<\bar\eps$, there is a set $F\subseteq\R^N$ such that
\begin{align*}
F\triangle E \subset\subset B\,, && |F|_f - |E|_f = \eps\,, && P_f(F) \leq P_f(E) + C |\eps|^\beta\,.
\end{align*}
\end{definition}
We give also the following simple definition, which will only be used within the subsequent Definition~\ref{essbdd}.
\begin{definition}
A family $\big\{ U_\delta\big\}_{\delta>0}$ of open subsets of $\R^N$ is said \emph{well-decreasing} if one has that $\R^N\setminus U_\delta$ is bounded for every $\delta>0$, and for any measurable set $E\neq \emptyset$ of finite perimeter one has $\H^{N-1}\big(\partial^* E\cap U_\delta\big) > 0$ for some arbitrarily small $\delta$.
\end{definition}
Notice that, for instance, the sets $U_\delta=\R^N \setminus \overline{B_\delta}$ form a well-decreasing family; more in general, for a family of open sets with $\R^N\setminus U_\delta$ bounded, to be well-decreasing it is enough that $\H^{N-1}\big(\R^N \setminus \bigcup_{\delta>0} U_\delta\big)=0$.\par
In the paper, we will always assume one of the following hypotheses on $f$.
\begin{definition}\label{essbdd}
The l.s.c. function $f:\R^N\to \R^+$ is said to be \emph{essentially bounded} if there exist a well-decreasing family $\big\{U_\delta\big\}_{\delta>0}$ and constants $M=M(\delta)$ such that $\frac 1M < f(x) < M$ for every $x\in U_\delta$. Analogously, $f$ is said to be \emph{essentially $\alpha$-H\"older} for some $0\leq \alpha\leq 1$ if $f$ is essentially bounded, and there exist a well-decreasing family $\big\{U_\delta\big\}_{\delta>0}$ and constants $M=M(\delta)$ such that $|f(x)-f(y)|\leq M |x-y|^\alpha$ for every $x,\,y\in U_\delta$. Observe that $f$ is essentially $\alpha$-H\"older with $\alpha=0$ if and only if it is essentially bounded.
\end{definition}

Notice that, clearly, if $f$ is bounded by above and below then it is also essentially bounded, and similarly if it is $\alpha$-H\"older continuous then it is also essentially $\alpha$-H\"older, hence we could have simply restricted our attention to standard bounded or H\"older densities. Nevertheless, we prefer to use the above slightly more complicate definitions for two reasons. First of all, this allows to consider also the typical class of examples where the density may vanish of explode at the origin, such as
\[
f(x)=\frac{|x|^p}{1+|x|^p}\,, \qquad \hbox{or}\qquad f(x)=1+\frac1{|x|^p}
\]
for $p>0$. Second, as one can see later, this choice does not effect at all any of the proofs, so we can obtain slightly stronger results for free.\par

Another comment deserves to be done concerning the choice of considering only densities which are (at least in the essential sense) bounded both from above and from below. In fact, this is unavoidable, since many examples (see for instance those in~\cite{MP}) enlighten that both existence and boundedness can easily fail otherwise, thus no general result can be obtained without this assumption.\par

We are now in position to claim our two main theorems.

\maintheorem{boundedness}{Boundedness of isoperimetric sets}{Assume that $f$ is essentially bounded and that $E$ is an isoperimetric set fulfilling the $\eps-\eps^\beta$ property, either with $\beta > \frac{N-1}{N}$ and some $C>0$, or with $\beta=\frac{N-1}N$ and \emph{every small} $C>0$. Then, $E$ is bounded.}

Notice that, in the above theorem, we do not need any continuity assumption on $f$: the sole $\eps-\eps^\beta$ property together with the essential boundedness of $f$ is enough to ensure the boundedness of any isoperimetric set $E$.

\maintheorem{eps-epsbeta}{The $\eps-\eps^\beta$ property}{Assume that $f$ is essentially $\alpha$-H\"older for some $0\leq\alpha\leq 1$. Then every set $E$ of locally finite perimeter fulfills the $\eps-\eps^\beta$ property with some $C$, where $\beta$ is defined by
\begin{equation}\label{defbeta}
\beta=\beta(\alpha,N) := \frac{\alpha+(N-1)(1-\alpha)}{\alpha+N(1-\alpha)}\,.
\end{equation}
Moreover, if $f$ is essentially bounded and continuous, then every set $E$ of locally finite perimeter still fulfills the $\eps-\eps^\beta$ property with $\beta=\frac{N-1}N$ and \emph{with any constant $C>0$}.
}

Notice that, once $N$ is fixed, $\beta$ is a continuous and strictly increasing function of $\alpha$ with $\beta=\frac{N-1}N$ for $\alpha=0$, and $\beta=1$ for $\alpha=1$, and moreover $\beta>\alpha$ for every $0\leq \alpha <1$. Notice also that Theorem~\ref{obvious} is an immediate consequence of the two above theorems.
\begin{remark}
It is essential to underline a very delicate point in the claim of Theorem~\mref{eps-epsbeta}, namely, that in the case when $f$ is just essentially bounded and continuous, the $\eps-\eps^\beta$ property with $\beta=\frac{N-1}N$ holds true \emph{for every constant $C>0$} (of course, when $C$ becomes smaller, so does the constant $\bar\eps$ in Definition~\ref{defepepbe}). This is of primary importance: in fact, as the proof of Theorem~\mref{boundedness} will enlighten, when a set fulfills the $\eps-\eps^\beta$ property with some $\beta>\frac{N-1}N$, then the constant $C$ is unessential; on the other hand, if $\beta=\frac{N-1}N$, then it is fundamental that the constant $C$ can be chosen arbitrarily small. And indeed, if $f$ is essentially bounded but not continuous then the $\eps-\eps^{\frac{N-1}N}$ property holds true, but not with any constant $C$, and hence one cannot apply Theorem~\mref{boundedness} and in fact an isoperimetric set can be unbounded, as we will show with the example in Section~\ref{example}.
\end{remark}

\subsection{Basic properties of sets of finite perimeter\label{sec:finper}}

Since a basic knowledge of the theory of sets of finite perimeter is needed for the proof of our results, we recall here very briefly what we are going to use. For a general tractation of this subject, and for the proof of all the claims of this section, the interested reader should refer for instance to~\cite{AFP}.\par

Let then $E\subseteq\R^N$ be a set of locally finite volume. We say that $E$ is \emph{of locally finite perimeter} if the characteristic function $\chi_E$ of $E$ is a $BV_{\rm loc}$ function. In other words, we require $\mu_E:=D\chi_E$ to be a vector valued and locally finite Radon measure. If $E$ is a set of locally finite perimeter, one defines the \emph{reduced boundary} $\partial^* E$ of $E$ as the set of all those points $x\in \R^N$ such that there exists a (necessarily unique) vector $\nu_E(x)\in \S^{N-1}$ with the property that
\[
\lim_{r \searrow 0} \frac{\big| E \cap B(x,r)\big|_{\eu}}{\big| B(x,r)\big|_\eu}=
\lim_{r \searrow 0} \frac{\Big| E \cap B(x,r)\cap \big\{ y:\, (y-x)\cdot \nu_E(x)<0\big\}\Big|_{\eu}}{\big| B(x,r)\big|_\eu}= \frac 12\,.
\]
The vector $\nu_E(x)$ is called the \emph{(measure theoretical) outer normal of $E$ at $x$}.\par
One can prove that $\mu_E$ coincides with the vector valued measure $\nu_E(x) \H^{N-1}\res \partial^* E$; finally, one says that the (Euclidean) perimeter of $E$ is defined as
\[
P_\eu (E) = | \mu_E | (\R^N) = \H^{N-1} \big( \partial^* E\big) = \int_{\partial^* E } 1\, d\H^{N-1}(x)\,. 
\]
It is easy to show that, if the set $E$ is regular enough, then this general notion of perimeter coincides with the usual perimeter, the reduced boundary coincides with the usual topological boundary, and the measure theoretical outer normal coincides with the usual outer normal. However, there exist sets of finite perimeter for which the topological boundary and the reduced boundary do not coincide; for instance, the set of the points in $\R^N$ with rational coordinates has an empty reduced boundary and so null perimeter, but its topological boundary is the whole $\R^N$. We conclude by recalling three classical results, that we will use extensively in the sequel.
\begin{theorem}[Blow-up Theorem]\label{blowupthm}
Let $E\subseteq\R^N$ and $x\in\partial^* E$. For every $\eps>0$, define the blow-up set $E_\eps := \frac 1\eps\, (E -x)$, and call $\mu_\eps:= \mu_{E_\eps}$ and $H=\{x\in \R^N:\, x\cdot \nu(x) <0\}$. Then, when $\eps\searrow 0$, one has that the sets $E_\eps$ converge to $H$ in the $L^1_{\rm loc}$ sense, while the measures $\mu_\eps$ (resp., $|\mu_\eps|$) weak* converge to the measure $\nu_E(x) \H^{N-1}\res \partial H$ (resp., $\H^{N-1}\res \partial H$).
\end{theorem}

To state the following result, the Vol'pert Theorem, we need a simple preliminary piece of notation.

\begin{definition}
Let $E\subseteq\R^N$ be a Borel set. We define the \emph{vertical section} of $E$ at any level $y\in\R^{N-1}$, and the \emph{horizontal section} of $E$ at any level $t \in\R$ as
\begin{align*}
E_y:=\{t\in \R:(y,t)\in E\}\,, && E^t:=\{y\in \R^{N-1}:(y,t) \in E\}\,.
\end{align*}
\end{definition}
\begin{theorem}[Vol'pert]\label{volpert}
Let $E$ be a set of locally finite perimeter. Then, for $\H^{N-1}$-a.e. $y\in \R^{N-1}$ the vertical section $E_y$ is a set of finite perimeter in $\R$, and $\partial^* (E_y) = \big(\partial^* E\big)_y$. Analogously, for $\H^1$-a.e. $t\in\R$ the horizontal section $E^t$ is a set of finite perimeter in $\R^{N-1}$, and $\partial^*(E^t)=(\partial^*E)^t$ up to an $\H^{N-2}$-negligible set.
\end{theorem}
The proof of this result can be found in~\cite{AFP,Volpert} for the case of the vertical sections, while the analogous for the horizontal sections (or in general, for sections of any codimension) is proved in~\cite{FMP,fusco,BCF}.

\begin{theorem}[Coarea formula]\label{coarea}
Let $E$ be a set of locally finite perimeter and denote, for $x\in\partial^*E$, the outer normal of $E$ at $x$ as $\nu_E(x)=\big(\nu'(x),\nu_N(x)\big)\in \R^{N-1}\times\R$. For every Borel function $g:\R^N\to \R^+$ it is
\[
\int_{\partial^* E}g(x)\sqrt{1-|\nu_N(x)|^2}\,d\H^{N-1}(x)=\int_{-\infty}^{+\infty}\int_{\partial^* E^t}g(y,t)\,d\H^{N-2}(y)\,dt\,,
\]
and analogously
\[
\int_{\partial^* E}g(x)\sqrt{1-|\nu'(x)|^2}\,d\H^{N-1}(x)=\int_{\R^{N-1}}\int_{\partial^* E_y}g(y,t)\,d\H^0(t)\,d\H^{N-1}(y)\,.
\]
\end{theorem}

\section{Proof of Theorem~\mref{boundedness}\label{bounded}}

This section is entirely devoted to give a proof of Theorem~\mref{boundedness}.

\begin{proof}[Proof (of Theorem~\mref{boundedness})]
Let $E$ be an isoperimetric set. Since $f$ is essentially bounded, we can take some $R'>0$ and $M>1$ such that $\frac 1M \leq f \leq M$ outside $B_{R'}$. Let us now use the assumption that $E$ fulfills the $\eps-\eps^\beta$ property, finding constants $C,\, R'',\,\bar\eps>0$ with $R''\geq R'$ such that, for every $0<\eps<\bar \eps$, there exists a set $F$ with $F=E$ outside the ball $B_{R''}$, and
\begin{align}\label{tkxepseps}
|F|_f = |E|_f + \eps\,, && P_f(F) \leq P_f(E) + C \eps^\beta\,;
\end{align}
in addition, we are also allowed to assume
\begin{align}\label{onlyforb=1/2}
C \leq \frac{2N\omega_N^{1/N}}{M^2} && \hbox{if } \beta = \frac{N-1}N \,.
\end{align}
Let us now define, for every $R> R''$,
\begin{align*}
\eps(R) := \big| E \setminus B_R\big|_f\,, && 
g(R) := \H^{N-1}_f\big( E \cap \partial B_R\big) = \int_{E \cap \partial B_R} f(x) \, d\H^{N-1}(x)\,;
\end{align*}
the function $R\mapsto \eps(R)$ is decreasing and goes to $0$ as $R$ goes to $\infty$. Moreover, observe that
\[
\eps(R) = \int_R^\infty \int_{\partial B_r} \chi_E(x) f(x) \, d\H^{N-1}(x) \, dr 
=\int_R^\infty \H^{N-1}_f\big( E \cap \partial B_R\big)\,dr= \int_R^\infty g(r)dr\,,
\]
hence $\eps\in {\rm W}^{1,1}_{\rm loc}(R'',+\infty)$ and $\eps'(R) = -g(R)$. Pick now any $R > R''$: we can consider a competitor $\widetilde E$ by cutting away the part of $E$ which is outside the ball $B_R$, and then using the $\eps-\eps^\beta$ property to recover the correct volume without increasing too much the perimeter. More precisely, first of all we notice that
\begin{equation}\label{starthere}
P_f \big( E \cap B_R\big) = P_f(E) - P_f\big(E\setminus B_R\big) + 2 \H^{N-1}_f \big( E \cap \partial B_R \big)
= P_f(E) - P_f\big(E\setminus B_R\big) + 2 g(R)\,.
\end{equation}
Then, we apply the standard Euclidean isoperimetric inequality and the bounds on $f$ to get
\[
P_f\big( E\setminus B_R\big) \geq \frac 1M\, P_\eu \big( E \setminus B_R\big) 
\geq \frac 1M\, N \omega_N^{1/N} \big| E \setminus B_R\big|_\eu^{\frac{N-1}N}
\geq \frac 1{M^{\frac{2N-1}N}}\, N \omega_N^{1/N} \eps(R)^{\frac{N-1}N}\,,
\]
which inserted into~(\ref{starthere}) gives
\begin{equation}\label{continue}
P_f \big( E \cap B_R\big) \leq P_f(E) - \frac 1{M^{\frac{2N-1}N}}\, N \omega_N^{1/N} \eps(R)^{\frac{N-1}N} + 2 g(R)\,.
\end{equation}
We apply now~(\ref{tkxepseps}) with $\eps=\eps(R)$ to find a set $F\subseteq \R^N$ with $|F|_f=|E|_f+\eps$ and $P_f(F)\leq P_f(E) + C\eps^\beta$. Since $F$ coincides with $E$ outside $B_{R''}$ and $R >R''$, we can finally define the competitor $\widetilde E$ as $F\cap B_R$. By construction we find
\[
\big|\widetilde E\big|_f = \big| E \cap B_R\big|_f + \eps(R) = |E|_f\,,
\]
while by~(\ref{continue}) it is also
\[
P_f\big(\widetilde E\big) = P_f\big( E \cap B_R \big) + P_f(F) - P_f(E)
\leq P_f(E) - \frac {N \omega_N^{1/N}}{M^{\frac{2N-1}N}}\, \eps(R)^{\frac{N-1}N} + 2 g(R)+ C \eps(R)^\beta\,.
\]
Since $E$ is an isoperimetric set and $\widetilde E$ has the same volume as $E$, we derive that for every $R>R''$ it must be $P_f(\widetilde E)\geq P_f(E)$, which implies
\begin{equation}\label{prebasic}
g(R) \geq \frac {N \omega_N^{1/N}}{2M^{\frac{2N-1}N}}\, \eps(R)^{\frac{N-1}N} - \frac C2 \,\eps(R)^\beta\,.
\end{equation}
We claim then what follows: there exist two positive constants $\gamma$ and $\hat\eps\leq \bar\eps$ such that
\begin{align}\label{basic}
g(R) \geq \gamma \eps(R)^{\frac{N-1}N} && \forall\, R>R'':\, \eps(R) \leq \hat\eps\,.
\end{align}
It is immediate to prove the validity of this estimate by considering separately the case $\beta=\frac{N-1}N$ and the case $\beta>\frac{N-1}N$. Indeed, in the first case~(\ref{basic}) is just an immediate consequence of~(\ref{prebasic}) and of the choice~(\ref{onlyforb=1/2}). On the other hand, if $\beta>\frac{N-1}N$, then $\eps^{\frac{N-1}N}\gg \eps^\beta$ when $\eps$ is small enough, then~(\ref{basic}) readily follows by~(\ref{prebasic}). Since $R\mapsto \eps(R)$ is a continuous and decreasing function, we can select $R'''\geq R''$ such that $\eps(R)\leq \hat \eps$ for every $R\geq R'''$.\par
Let us now directly show the boundedness of $E$: if $E\subseteq B_{R'''}$, then there is nothing to prove; otherwise, let $j\in \N$ be such that $\eps(R''')\geq 2^{-j}$, and for every $i\geq j$ let $R(i)$ be such that $\eps(R_i)= 2^{-i}$. Then, recalling~(\ref{basic}) and the fact that $\eps'=-g$, for every $i\geq j$ we can evaluate
\[\begin{split}
\frac 1{2^{i+1}} &= \frac 1{2^i} - \frac 1{2^{i+1}} 
= \eps\big(R_i\big) - \eps\big(R_{i+1}\big)
= \int_{R_i}^{R_{i+1}} g(R)\, dR
\geq \int_{R_i}^{R_{i+1}} \gamma \eps(R)^{\frac{N-1}N} \, dR
\geq \gamma \, \frac{R_{i+1}-R_i}{2^{(i+1)\,\frac{N-1}N}}\,,
\end{split}\]
from which we deduce that
\[
R_{i+1} - R_i \leq \frac{1}{\gamma}\, 2^{-\frac{i+1}N} \,.
\]
This immediately implies that, for every $\ell\geq j$,
\[
R_\ell \leq R_j + \frac 1 \gamma\, \sum_{i=j}^{\ell-1} \frac{1}{\big(2^{-\frac 1 N}\big)^{i+1}}
\leq R_j + \frac 1 \gamma\, \sum_{i=j}^{\infty} \frac{1}{\big(2^{-\frac 1 N}\big)^{i+1}} =: R_\infty < +\infty\,,
\]
and in turn this implies that $\eps(R)=0$ for every $R\geq R_\infty$, that is, $E\subseteq B_{R_\infty}$ is bounded.
\end{proof}

\section{Proof of Theorem~\mref{eps-epsbeta}\label{holder}}

In this section we give the proof of our main result, Theorem~\mref{eps-epsbeta}. Since the proof is quite involved, we have divided it for simplicity in three parts and several steps.

\begin{proof}[Proof (of Theorem~\mref{eps-epsbeta})]
Let us consider a function $f$ as in the claim, and an isoperimetric set $E$. In the first part, we will show that the $\eps-\eps^\alpha$ property holds. Since $\alpha<\beta$ unless $\beta=\alpha=1$, the property is in fact weaker than what we need; nevertheless, we prefer to start with this somehow easier case, because the proof of the stronger $\eps-\eps^\beta$ property, which will be done in the second part, will be a careful modification of the same argument. And in turn, also the case when $f$ is only essentially continuous will eventually be treated, in the third part, with the same strategy. By Definition~\ref{essbdd}, there exists an open set $U\subseteq\R^N$ with $U\cap \partial^*E\neq \emptyset$ and such that, for a suitable $M>1$, one has 
\begin{align*}
\frac 1M < f(x) < M \quad \forall\, x\in U\,, &&
|f(x)-f(y)|\leq M |x-y|^\alpha\quad \forall\, x\,,\, y \in U\,.
\end{align*}
Let $\bar x$ be a point in $ U\cap \partial^* E$. We can assume for simplicity that $\bar x$ coincides with the origin in $\R^N$, and that the outer normal of $E$ at $\bar x$ is the vertical direction $\nu(\bar x)=(0,1)\in\R^{N-1}\times \R$.

\part{I}{The $\eps-\eps^\alpha$ property.}
We start considering the case when $f$ is essentially $\alpha$-H\"older, and prove the $\eps-\eps^\alpha$ property; the proof is divided in many steps.

\step{(i)}{Choice of the cube $Q^N$.}
In this first step, we will select a suitably small constant $a$, and from now on we will restrict our attention to the cube $Q^N=(-a/2,a/2)^N$, which is entirely contained inside $U$ as soon as $a\ll 1$. Let us denote by $Q=(-a/2,a/2)^{N-1}$ the horizontal cube, and by $\varphi:\R^{N}\rightarrow \R^N$ the constant vector field $\varphi \equiv (0,1) \in \R^{N-1}\times \R$; let moreover $\rho>0$ be a sufficiently small constant, that will be precised later. We aim to choose $a>0$ such that $Q^N$ is contained in the open set $U$ defined above, and moreover all the following properties hold:
\begin{gather}
1-\rho \leq \frac{\H^{N-1}\big(\partial^*E \cap Q^N\cap \{-a\rho<x_N<a\rho\}\big)}{a^{N-1}} \leq 1+\rho\,,\label{P2}\\
\frac{\H^{N-1}\big(\partial^*E \cap Q^N\setminus\{-a\rho<x_N<a\rho\}\big)}{a^{N-1}} \leq \rho\,,\label{P3}\\
\frac{\H^N\big(E \cap Q^N\cap\{x_N<0\}\big)}{a^N/2} \geq 1-\rho\,,\label{V1}\\
\frac{\H^N\big(E \cap Q^N\cap\{x_N>0\}\big)}{a^N/2} \leq \rho\,,\label{V2}\\
\int_{Q^N} \varphi \cdot d\mu_E \geq (1-\rho)a^{N-1}\,, \label{vect}\\
\H^{N-2}\Big(\partial^*E\cap\big(\partial Q\times (-a/2,a/2)\big)\Big) \leq 2^{N+1} a^{N-2}\,.\label{bordo}
\end{gather}
We show now that such a choice of $a$ is possible. Indeed, the first five conditions~\eqref{P2}--\eqref{vect} are true for \emph{every} $a$ small enough, say $a\leq \bar a$, as a direct consequence of the blow-up Theorem~\ref{blowupthm}. It remains then only to show that there exists \emph{some} $a\leq \bar a$ satisfying also condition~(\ref{bordo}). To do so observe that, also by~(\ref{P2}) and~(\ref{P3}),
\[\begin{split}
\int_{\bar a/4}^{\bar a/2}\H^{N-2}\Big(\partial^*E \cap \big(\partial (-t,t)^{N-1}\times (-\bar a/2,\bar a/2)\big)\Big)dt
&\leq \H^{N-1}\Big(\partial^*E\cap \big(-\bar a/2, \bar a/2\big)^N\Big)\\
&\leq \big( 1+2\rho\big) \bar a^{N-1}\,.
\end{split}\]
Therefore, there exists $a\in (\bar a/2,\bar a)$ for which
\[\begin{split}
\H^{N-2}\Big(\partial^*E&\cap\big(\partial Q\times (-a/2,a/2)\big)\Big)\leq
\H^{N-2}\Big(\partial^*E \cap \big(\partial (-a/2,a/2)^{N-1}\times (-\bar a/2,\bar a/2)\big)\Big)\\
&\leq 4\,\frac{\big( 1+2\rho\big) \bar a^{N-1}}{\bar a}\leq 5 \bar a^{N-2}\leq 2^{N+1} a^{N-2}\,.
\end{split}\]
Notice that, thanks to Vol'pert Theorem, without loss of generality we can assume
\begin{equation}\label{noticethat}
\partial^*E \cap\big(\partial Q\times (-a/2,a/2)\big)=\partial^*\Big(E\cap\big(\partial Q\times (-a/2,a/2)\big)\Big)\qquad \H^{N-2}-a.e.\,.
\end{equation}
We have then found some $a\leq \bar a$ for which also~(\ref{bordo}) holds true. This concludes the first step.

\step{(ii)}{Definition of $A$, $B$, $G$ and $\Gamma$.}
In this step, we subdivide the horizontal cube $Q$ into four sets $A$, $B$, $G$ and $\Gamma$, depending on the properties of $\partial^*(E\cap Q^N)_{x'}$. Since in the whole proof we are concentrated only on what happens inside $Q^N$, we will always consider the horizontal and vertical sections inside the cube, even without specifying it; in other words, we will write $E_y$ or $E^t$ (respectively $\partial^*E_y$ or $\partial^*E^t$) instead of $(E \cap Q^N)_y$ or $(E\cap Q^N)^t$ (respectively $\partial^*(E \cap Q^N)_y$ or $\partial^*(E\cap Q^N)^t$). This is a slight abuse of notation, but it will simplify a lot the formulas in the rest of the proof. The sets are the following
\begin{align*}
A &:=\{ x' \in Q: \partial^*(E_{x'})\neq (\partial^* E)_{x'} \}\,,\\
B &:=\{x'\in Q\setminus A: \H^0(\partial^* E_{x'})=0\}\,,\\
G& :=\{x'\in Q \setminus A: \H^0(\partial^* E_{x'})=1,\,\partial^* E_{x'}\subseteq (-a\rho,a\rho),\,E_{x'} \subseteq (-a/2, a\rho) \}\,,\\
\Gamma&:=Q\setminus(A\cup B\cup G)\,.
\end{align*}
Let us briefly discuss the meaning of these sets: thanks to Step~(i), we can imagine $E\cap Q^N$ to be close to the half-cube $Q^N\cap \{x_N<0\}$, thus we expect the vertical sections $E_{x'}$ to be close to $(-a/2,0)$. The ``good'' set $G$ is precisely the set of those $x'\in Q$ for which this holds, namely, $E_{x'}$ is a ``lower'' segment starting at $-a/2$ and ending between $-a\rho$ and $a \rho$. All the other $x'\in Q$ are then contained in the ``bad'' sets $A$, $B$ and $\Gamma$. More precisely, $A$ collects those $x'$ for which Vol'pert Theorem does not hold true (keep in mind that we know by Theorem~\ref{volpert} that $A$ is $\H^{N-1}$ negligible, but this does not imply that the sections corresponding to $A$ do not carry perimeter!). Instead, $B$ is the set of the sections which have no boundary, thus are either the full segment $(-a/2,a/2)$, or they are empty. Finally, $\Gamma$ is the set of the sections having some boundary, but not contained in $G$. Observe that this can happen for several different reasons: if $\partial^* E_{x'}$ contains exactly one point, then either this point is not between $-a\rho$ and $a\rho$, or $E_{x'}$ is an ``upper'' segment starting between $-a\rho$ and $a\rho$, and ending at $a/2$. On the other hand, if $\partial^* E_{x'}$ has more than one point, then the points can be finitely many or infinitely many. We further subdivide $\Gamma$ in four subsets according to the above possibilities, namely, we define
\begin{align*}
\Gamma_0 &:=\big\{x' \in \Gamma: \H^0(\partial^*E)_{x'}=1, \partial^*E_{x'}\notin (-a\rho, a\rho)\big\}\,,\\
\Gamma_1 &:=\big\{x' \in \Gamma \setminus \Gamma_0: \H^0(\partial^*E)_{x'}=1\big\}\,,\\
\Gamma_2 &:=\big\{x' \in \Gamma\setminus (\Gamma_0 \cup \Gamma_1):\partial^*E_{x'}\,\mbox{contains a finite number of points} \big\}\,,\\
\Gamma_3 &:=\Gamma\setminus \big(\Gamma_0 \cup \Gamma_1\cup \Gamma_2\big)\,.
\end{align*}
The aim of this step is to show that $G$ fills a big portion of $Q$, and that the perimeter of $E$ in the sections not belonging to $G$ is extremely small. Let us start by observing that, thanks to~\eqref{vect}, one has
\begin{equation}\label{gamma}\begin{split}
(1-\rho)a^{N-1}\leq \int_{Q^N} \varphi\cdot d\mu_E
= &\int_{A\times (-a/2,a/2)} \varphi\cdot d\mu_E+\int_{B\times (-a/2,a/2)} \varphi\cdot d\mu_E\\
&\hspace{20pt}+ \int_{G\times (-a/2,a/2)}\varphi\cdot d\mu_E+\int_{\Gamma\times (-a/2,a/2)}\varphi\cdot d\mu_E\,.
\end{split}\end{equation}
We have now to estimate each of the terms in the right-hand side of last inequality. First, since by construction $d\mu_E=0$ on the set $B\times (-a/2,a/2)$, we have
\begin{equation}\label{est:B}
\int_{B\times (-a/2,a/2)} \varphi\cdot d\mu_E=0\,.
\end{equation}
We address now the integral on $\Gamma_0 \times (-a/2,a/2)$. Recall that, as already observed, if $x'\in\Gamma_0$ and $(x', x_N)\in \partial^*E$, then $x_N \notin (-a\rho, a\rho)$. Therefore, using~\eqref{P3}, we get 
\begin{equation}\label{gamma-1}\begin{split}
\int_{\Gamma_0\times(-a/2,a/2)}\varphi\cdot d\mu_E&
\leq \int_{\Gamma_0\times(-a/2,a/2)} |\varphi| \, d | \mu_E|
= \H^{N-1}(\partial^*E\cap\{(x',x_N):x'\in \Gamma_0\})\\
&\leq \H^{N-1}(\partial^*E\cap Q^N\setminus \{-a\rho<x_N<a\rho\})\leq \rho a^{N-1}\,.
\end{split}\end{equation}
Concerning $A$, we just recall that by Vol'pert Theorem~\ref{volpert} it is
\begin{equation}\label{est:A}
\H^{N-1}(A)=0\,.
\end{equation}
Let us pass now to $\Gamma_3$; as already observed, for every $x'\in \Gamma_3$ the set $\partial^*(E_{x'})= (\partial^*E)_{x'}$ contains infinitely many points. Then, since for any $K\geq 1$ it is clearly
\[
\H^{N-1}\Big(\big\{x'\in Q\setminus A:\,\H^0(\partial^*E_{x'})\geq K\big\}\Big)\leq\frac{1}{K} \,\H^{N-1}(\partial^*E\cap Q^N)\,,
\]
by sending $K\rightarrow \infty$ we derive
\begin{equation}\label{gamma-2c}
\H^{N-1}(\Gamma_3)=0\,.
\end{equation}
Thanks to~\eqref{est:A} and~\eqref{gamma-2c}, the coarea formula (Theorem~\ref{coarea}) directly gives
\begin{equation}\label{gamma-2-final}\begin{split}
\int_{(\Gamma_3 \cup A)\times (-a/2,a/2)}\varphi\cdot d\mu_E&= \int_{Q^N\cap\partial^*E \cap \{x'\in \Gamma_3\cup A\}} (0,1)\cdot \nu_E\, d\H^{N-1} \\
&=\int_{Q^N\cap \partial^*E \cap \{x'\in \Gamma_3 \cup A\}}\sqrt{1-|\nu'_E(x)|^2}\,d\H^{N-1}(x)\\
&= \int_{Q^N\cap\Gamma_3 \cup A} \left(\int_{\partial^*E_{x'}}1 \,d\H^0(x_N)\right)\, d\H^{N-1}(x')=0\,.
\end{split}\end{equation}
We address now $\Gamma_1$. Recall that, by definition, if $x' \in \Gamma_1$ then $\partial^*E_{x'}=\{p\}$ with $p=p(x) \in (-a\rho, a\rho)$. Call then $\widetilde \Gamma_1$ the set of those $x'\in \Gamma_1$ for which $|\nu_E'(x',p)|<1$, that is, the levels $x'$ such that the outer normal at $(x',p)$ is \emph{not} horizontal. We remark the well known fact that $\H^{N-1}\big(\Gamma_1\setminus \widetilde\Gamma_1\big)=0$. Using again the coarea formula, denoting by $\delta_p$ the Dirac mass at $p \in \R$ we find
\[\begin{split}
|\mu_E| \res \big(\widetilde\Gamma_1\times (-a/2,a/2)\big)
&=\frac{1}{\sqrt{1-|\nu'_E(x',p)|^2}}\,\delta_p\otimes \H^{N-1} \res \widetilde\Gamma_1\\
&= \frac{1}{|\nu_E(x',p)\cdot (0,1)|}\,\delta_p\otimes \H^{N-1} \res \widetilde\Gamma_1\,.
\end{split}\]
Hence we have
\begin{equation}\label{gamma-0}\begin{split}
\int_{ \Gamma_1\times (-a/2,a/2)} \varphi \cdot d\mu_E 
&=\int_{Q} (0,1)\cdot \nu_E(x)\,d|\mu_E|\res \big(\Gamma_1\times(-a/2,a/2)\big)\\
&=\int_{Q} (0,1)\cdot \nu_E(x)\,d|\mu_E|\res \big(\widetilde\Gamma_1\times(-a/2,a/2)\big)\\
&=\int_{ \widetilde\Gamma_1} \frac{(0,1)\cdot \nu_E(x',p)}{| (0,1)\cdot \nu_E(x',p)|}\, d\H^{N-1}(x')
=-\H^{N-1}(\widetilde\Gamma_1)=-\H^{N-1}(\Gamma_1)\,.\hspace{-15pt}
\end{split}\end{equation}
Note that the ``$-$'' sign comes from that fact that $\nu_E(x',p)$ has clearly a negative last component for every $x'\in \widetilde \Gamma_1$.\par
The very same argument used for $\Gamma_1$ can be repeated for $G$, recalling that for every $x' \in G$ one has that $\partial^*E_{x'}=\{q\}$ with some $q=q(x)\in (-a\rho, a\rho)$. Therefore, since $\nu_E(x',q)$ has a positive last component, in place of~\eqref{gamma-0} we find now
\begin{equation}\label{HG}
\int_{G\times (-a/2,a/2)}\varphi \cdot d\mu_E = \H^{N-1}(G)\,.
\end{equation}
Finally, we address $\Gamma_2$. First of all, recall that $\H^0(\partial^*E_{x'})\geq 2$ for almost every $x \in \Gamma_2$. Thus
\begin{equation}\label{gamma2}
\H^{N-1}(\partial^*E \cap \{(x',x_N):x' \in \Gamma_2\})\geq \int_{\Gamma_2} \H^0(\partial^*E_{x'})\,d\H^{N-1}(x')\geq 2 \H^{N-1}(\Gamma_2)\,.
\end{equation}
Moreover, by definition, for every $x'\in \Gamma_2$ there exist $k\in \N$ and $p_i\in (-a/2,a/2)$, for $i=1,...,k$, such that $\partial^*E_{x'}=\bigcup_{i=1}^k p_i$. Let us call again $\widetilde \Gamma_2$ the set of those $x'$ for which all the corresponding $p_i$ have a non-horizontal outer normal, where again $\H^{N-1}\big(\Gamma_2\setminus\widetilde\Gamma_2\big)=0$. Using the coarea formula exactly as we did for $\Gamma_1$ and $G$, we have that 
\[
|\mu_E| \res(\widetilde\Gamma_2\times (-a/2,a/2))=\alpha_{x'} \otimes \H^{N-1} \res \widetilde\Gamma_2\,,
\]
where
\[
\alpha_{x'}=\sum_{i=1}^k \frac{1}{|\nu_E(x',p_i)\cdot (0,1)|}\,\delta_{p_i}\,.
\]
Therefore we have, analogously as in~\eqref{gamma-0} or~\eqref{HG},
\begin{equation}\label{integrand}
\int_{ \Gamma_2\times (-a/2,a/2)} \varphi \cdot d\mu_E 
=\int_{ \widetilde\Gamma_2} \sum_{i=1}^k \frac{(0,1)\cdot \nu_E(x',p_i)}{| (0,1)\cdot \nu_E(x',p_i)|} \,d\H^{N-1}(x') \,.
\end{equation}
Recall that, in the last expression, the integer $k$ and the points $p_i$ depend on $x'$. Observe now that, in the right hand side of last equation, the integrand is always a finite sum of $\pm 1$; let us discuss carefully the signs. For $x' \in \widetilde\Gamma_2$, we have that $E_{x'}=\bigcup_{j=1}^h (b_j,c_j)$ is a finite union of segments. Moreover, by construction $\{p_i\}_{i=1,...,k}=\{b_j\}_{j=1,...,h}\cup \{c_j\}_{j=1,...,h}\setminus \{-a/2,a/2\}$, since we are working within the open cube $Q^N$. In addition, one has that $\displaystyle \frac{(0,1)\cdot \nu_E(x',b_j)}{|(0,1)\cdot \nu_E(x',b_j)|}=-1$ for every $j=1,...,h$ such that $b_j\neq -a/2$, since the normal vector at any point $(x',b_j)$ has a negative last component; similarly, we have $\displaystyle \frac{(0,1)\cdot \nu_E(x',c_j)}{|(0,1)\cdot \nu_E(x',c_j)|}=1$ for every $j=1,...,h$ such that $c_j\neq a/2$. In conclusion, for every $x' \in \Gamma_2$ the value of the integrand in~\eqref{integrand} is either $-1$, or $0$, or $1$, depending whether or not $b_1=-a/2$ or $c_h=a/2$. Therefore we have, also recalling~\eqref{gamma2},
\begin{equation}\label{est-gamma3}
\int_{ \Gamma_2\times (-a/2,a/2)}\varphi \cdot d\mu_E
\leq \H^{N-1}(\Gamma_2)
\leq \frac{\H^{N-1}(\partial^*E\cap \{(x',x_N):x' \in \Gamma_2\})}{2}\,.
\end{equation}
Combining~\eqref{gamma}, \eqref{est:B}, \eqref{gamma-1}, \eqref{gamma-2-final}, \eqref{gamma-0}, \eqref{HG}, and~\eqref{est-gamma3}, we have
\begin{equation}\label{upper}\begin{split}
\H^{N-1}&\big(\partial^*E \cap \{(x',x_N)|x'\in G\}\big)+ \frac{1}{2}\, \H^{N-1}\big(\partial^*E\cap\{(x',x_N):x'\in \Gamma_2\}\big) \\
&\geq \H^{N-1}(G) + \frac{1}{2}\, \H^{N-1}\big(\partial^*E\cap\{(x',x_N):x'\in \Gamma_2\}\big) \\
&=\int_{G\times (-a/2,a/2)} \varphi \cdot d\mu_E +\frac{1}{2}\, \H^{N-1}(\partial^*E\cap\{(x',x_N):x'\in \Gamma_2\}) \geq (1-2\rho)a^{N-1}\,.
\end{split}\end{equation}
On the other hand, using~\eqref{P2} and~\eqref{P3} of we get the upper bound
\begin{equation}\label{lower}\begin{split}
\H^{N-1}(\partial^*E \cap \{(x',x_N):x'\in G\}) + \H^{N-1}(&\partial^*E \cap\{(x',x_N):x'\in \Gamma_2\})\\
&\leq \H^{N-1}(\partial^*E \cap Q^N)\leq (1+2\rho)a^{N-1}\,.
\end{split}\end{equation}
Combining together~\eqref{upper} and~\eqref{lower}, we deduce
\begin{equation}\label{gamma-1a}
\H^{N-1}(\partial^*E\cap\{(x',x_N):x'\in \Gamma_2\}) \leq 8\rho a^{N-1}\,.
\end{equation}
This, together with~\eqref{upper} again, implies that
\begin{equation}\label{est:G}
\H^{N-1}(\partial^*E\cap\{(x',x_N):x' \in G\})\geq \H^{N-1}(G) \geq (1-6\rho)a^{N-1}\,.
\end{equation}
Finally, this last estimate implies on one hand, since $A\cup G\cup B\cup \Gamma=Q$, that
\begin{equation}\label{B-gamma}
\H^{N-1}(A)+\H^{N-1}(B)+\H^{N-1}(\Gamma)\leq 6\rho a^{N-1}\,,
\end{equation}
and on the other hand, recalling~\eqref{P2} and~\eqref{P3}, that
\begin{equation}\label{est-gamma}
\H^{N-1}\Big(\partial^*E \cap\big(A \cup B\cup\Gamma \big)\times (-a/2,a/2)\Big)\leq 8 \rho a^{N-1}\,.
\end{equation}

\step{(iii)}{Definition of $\sigma^-$, $\sigma^+$ and $F$.}
\noindent\textbf{Definition of $\sigma^+$}. 
Let $\bar\delta\ll \rho a$ be a fixed number; we can take $\displaystyle H:=\left[\frac{a(1/2-3\rho)}{\bar\delta}\right]$ disjoint horizontal strips $S_i:=Q\times(\sigma_i, \sigma_i + \bar\delta)\subseteq Q^N$ with $a\rho< \sigma_i< a/2 - 2\rho$ for every $1\leq i\leq H$.
By assumptions~\eqref{P3} and~\eqref{V2}, we have
\[
\begin{split}
\sum_{i=1}^H a\H^{N-1}&\big(\partial^*E\cap S_i\big)+\H^{N}\big(E\cap S_i\big) \\
&\leq a \H^{N-1}\Big(\partial^*E\cap \big( Q\times [a\rho,a/2)\big)\Big)+ \H^{N}\Big(E\cap\big( Q\times [a\rho,a/2)\big)\Big) \leq 2\rho a^{N}\,.
\end{split}
\]
Therefore there exists $\overline i \in \{1,...,H\}$ such that
\begin{equation}\label{PV+}
a\H^{N-1}(\partial^*E\cap S_{\overline i})+\H^N(E\cap S_{\overline i})\leq \frac{2\rho a^N}{H}\leq 6 \bar\delta \rho a^{N-1}\,,
\end{equation}
recalling that by definition $\displaystyle H\geq \frac{a}{3\bar\delta}\,.$ We set $\sigma^+:=\sigma_{\overline i}$, for such an $\overline i$.

\noindent\textbf{Definition of $\sigma^-$}. 
We now select a horizontal level $\sigma^-\in (-a/2, -a\rho)$ such that $\partial^*(E^{\sigma^-})=(\partial^*E)^{\sigma^-}$ in the $\H^{N-2}$ sense, and
\begin{equation}\label{P-}
\frac{\H^{N-2}\big(\partial^*E^{\sigma^-}\big)}{a^{N-2}}\leq 3\rho\,.
\end{equation}
To show that this is possible, we apply~\eqref{P3} and Vol'pert Theorem~\ref{volpert} to obtain that
\[\begin{split}
\rho &\geq \frac{\H^{N-1}\big(\partial^* E\cap Q^N \cap \{x_N< -a\rho\}\big)}{a^{N-1}}
\geq \int_{-a/2}^{-a\rho} \frac{\H^{N-2}(\partial^*E^t)}{a^{N-1}}\,dt\\
&=\bigg(\frac 12 -\rho \bigg)\intmed_{-a/2}^{-a\rho} \frac{\H^{N-2}(\partial^*E^t)}{a^{N-2}}\,dt
\geq \frac 13\intmed_{-a/2}^{-a\rho} \frac{\H^{N-2}(\partial^*E^t )}{a^{N-2}}\,dt\,,
\end{split}\]
from which the existence of some $\sigma^-$ satisfying~\eqref{P-} immediately follows.

\noindent\textbf{Definition of $F$}. We want now to construct the competitor $F$. To do so, we take a constant $\bar\delta/(4M^2)\leq \delta\leq \bar\delta$, and we define the set $F=F(\delta)$ as
\[
x\in F \Longleftrightarrow \left\{\begin{array}{ll}
x\in E \setminus Q^N \,,\\
x \in E\cap Q^N\cap \big(\{x_N\leq \sigma^-\} \cup \{x_N> \sigma^+ + \delta\big)\,,\\
(x',\sigma^-)\in E\cap Q^N\ \mbox{and}\ \sigma^-< x_N\leq \sigma^-+\delta\,, \\
(x',x_N-\delta)\in E\cap Q^N\ \mbox{and}\ \sigma^-+\delta<x_N \leq \sigma^+ +\delta\,.
\end{array}\right.
\]
In words, we eliminate the intersection of $E$ with the strip $Q\times (\sigma^+,\sigma^++\delta) \subseteq S_{\bar i}$, and we move up of a distance $\delta$ all the part of $E$ between the levels $\sigma^-$ and $\sigma^+$. Notice that by definition $\sigma^+ + \delta < a/2$, so nothing changes outside of $Q^N$.

\step{(iv)}{Evaluation of the volume and perimeter of $F$.}
We are now ready to evaluate the volume and the perimeter of the set $F$, in order to obtain the $\eps-\eps^\alpha$ property. 

\noindent\textbf{Volume}. By the definition of $F$, it is easy to expect that its volume should equal that of $E$ plus something similar to $a^{N-1}\delta$, since we are moving up of a distance $\delta$ a set within the cube of $(N-1)$-dimensional volume equal to $a^{N-1}$. This is exactly what we are going to prove, but some care is required since, in passing from $E$ to $F$, we could also lose some volume, basically for two reasons. First, because we are eliminating the strip $Q\times (\sigma^+,\sigma^++\delta) \subseteq S_{\bar i}$ (with which $E$ has a small intersection, though). Second, because the density is not constant, and then there is in principle the risk of moving the mass where the density is lower.

Let us define the sets $E^+:=F \setminus E$, $E_1^-:=E \setminus F \cap \big(Q \times (\sigma^-, \sigma^+)\big)$, and $E_2^-:=E\setminus F \cap \big(Q \times (\sigma^+, \sigma^+ + \delta)\big)$, so that $F=E \cup E^+ \setminus (E_1^- \cup E_2^-)$. By construction we have $E^+ \cap E = \emptyset$, 
$E_1^- \cap E_2^- = \emptyset$, and 
$E_1^- \cup E_2^- \subseteq E$, thus
\begin{equation}\label{est:Q}
|F|_f-|E|_f=|E^+|_f-|E^-_1|_f-|E_2^-|_f\,.
\end{equation}
Let us estimate the terms on the right-hand side of this equality, starting with $|E_2^-|$. Using that $f\leq M$ in $Q^N$ and~\eqref{PV+}, and recalling that $\bar\delta/(4M^2)\leq \delta\leq \bar\delta$, we have that
\begin{equation}\label{E2-}\begin{split}
\big|E^-_2\big|_f&\leq 
\Big|E\cap \big(Q\times (\sigma^+,\sigma^+ +\delta\big)\Big|_f
\leq\big|E\cap S_{\bar i}\big|_f
\leq M \H^N\big(E\cap S_{\bar i}\big)
\leq 6 \rho M a^{N-1} \bar\delta\\
&\leq 24 \rho M^3a^{N-1} \delta\,.
\end{split}\end{equation}
We pass now to $E^+$. Observe that $E^+ = \bigcup_{x' \in Q} E^+_{x'} \supseteq \bigcup_{x' \in G} E^+_{x'}$ and that, if $x' \in G$, then $E^+_{x'}$ is a segment of length $\delta$. Therefore, using~\eqref{est:G} and recalling that $f\geq 1/M$ on $Q^N$, we deduce
\begin{equation}\label{E+}
|E^+|_f\geq \frac{1}{M}\, \delta \H^{N-1}(G)\geq \frac{1}{M}\, \delta a^{N-1}(1-6\rho)\,.
\end{equation}
Since we need also an upper bound for $|E^+|_f$, we consider separately the sets $G$ and $Q\setminus G$.
In $G$ we have
\begin{equation}\label{E+G}
\big|E^+ \cap (G\times (-a/2,a/2))\big|_f \leq M \delta \H^{N-1}(G)\leq M\delta a^{N-1}\,.
\end{equation}
Recall now that $Q\setminus G=A\cup B\cup\Gamma$. By definition of $B$, if $x'\in B$ then $E_{x'}$ is either empty or the whole segment $(-a/2,a/2)$: in both cases, $E^+_{x'}=\emptyset$. Therefore, recalling also~\eqref{est:A} and~\eqref{gamma-2c}, we have
\begin{equation}\label{E+B}
\Big|E^+ \cap \big((A\cup B\cup \Gamma_3)\times (-a/2,a/2)\big)\Big|_f=0\,.
\end{equation}
Finally, observe that
\[
\mbox{for every}\,x' \in \Gamma_0 \cup \Gamma_1 \cup \Gamma_2\,,\qquad
\H^1\Big((E^+)_{x'} \cup (E_1^-)_{x'}\Big)\leq \delta \H^0(\partial^*E_{x'})\,.
\]
Therefore, using~\eqref{est-gamma}, we deduce that
\begin{equation}\label{VE2}\begin{split}
\Big|\big(E^+ \cup E_1^-\big) &\cap \Big(\big(\Gamma_0 \cup \Gamma_1 \cup \Gamma_2\big)\times (-a/2,a/2)\Big)\Big|_f\\
&\leq M\int_{\Gamma_0 \cup\Gamma_1 \cup \Gamma_2} \H^1\Big((E^+)_{x'}\cup (E_1^-)_{x'}\Big)\,d\H^{N-1}(x') \\
& \leq M\delta \int_{\Gamma_0 \cup\Gamma_1\cup \Gamma_2}\H^0(\partial^*E_{x'})\,d\H^{N-1}(x')\\
& \leq M\delta\H^{N-1}\Big(\partial^*E \cap\big(\Gamma_0 \cup\Gamma_1\cup\Gamma_2\big)\times (-a/2,a/2)\Big)
\leq 8\rho M \delta a^{N-1}\,.
\end{split}\end{equation}
It remains to estimate $|E^-_1|$. Observe that if $x' \in B\cup G$ then $(E^-_1)_{x'}=\emptyset$: then, arguing exactly as we did to get~(\ref{E+B}), using~\eqref{est:A} and~\eqref{gamma-2c}, we have
\begin{equation}\label{VE3}
\Big| E^-_1 \cap \big((A\cup B \cup G \cup \Gamma_3) \times (-a/2,a/2)\big)\Big|_f=0\,.
\end{equation}
We have now all the ingredients to estimate $|F|_f-|E|_f$ both from above and below, thanks to~\eqref{est:Q}. Indeed, on one hand, combining~\eqref{E+}, \eqref{VE2}, \eqref{VE3} and~(\ref{E2-}), and up to take $\rho$ sufficiently small, we get
\begin{equation}\label{lowest}
|F|_f-|E|_f\geq \delta a^{N-1}\left(\frac{1}{M}(1-6\rho) -8\rho M-24M^3\rho\right)\geq \frac{\delta a^{N-1}}{2M}\,.
\end{equation}
On the other hand, combining~\eqref{E+G}, \eqref{E+B}, and~\eqref{VE2}, we also find
\begin{equation}\label{uppest}
|F|_f-|E|_f\leq |E^+|_f\leq M\delta a^{N-1}(1+8\rho)\leq 2M\delta a^{N-1}\,.
\end{equation}

\noindent\textbf{Perimeter}. We are then left to find an upper bound for $P_f(F)-P_f(E)$ in terms of $\delta$. This will be the only point in this Part where we are going to use the $\alpha$-H\"older assumption on $f$.\par

We start pointing out that the change in perimeter has four contributions. First, since we move upwards the set $E$ of a distance $\delta$ inside the cube $Q^N$, on the lateral boundary $\partial Q\times (-a/2,a/2)$ we are adding a surface $T_1^+$ of ``height'' $\delta$, namely,
\[
T_1^+:=\big(\partial^* F \setminus \partial^* E\big)\cap \big(\partial Q \times (-a/2,a/2)\big)\,.
\]
 Second, since in the strip $Q\times (\sigma^-,\sigma^- + \delta)$ the set $F$ is defined as $F=E^{\sigma^-}\times (\sigma^-,\sigma^-+\delta)$, then we are creating some new surface $T_2^+$ as soon as $\partial^*E^{\sigma^-}$ is not empty. More precisely, we set
\[
T_2^+:=\partial^*E^{\sigma^-} \times (\sigma^-, \sigma^-+\delta)\,.
\]
Third, since we are cutting away the set $E\cap \big(Q \times (\sigma^+,\sigma^+ + \delta)\big)$, then we are removing some surface $T_3^-$ in the strip, but at the same time we might also create some new surface $T_3^+$ at the level $\sigma^++\delta$. Hence, we call
\begin{align*}
T_3^-:=\partial^*E \cap \big(Q \times (\sigma^+,\sigma^+ + \delta)\big)\,, && T_3^+:=\pi'\big(T_3^-\big)\,,
\end{align*}
being $\pi': Q^N \rightarrow Q\times \big\{x_N=\sigma^+ + \delta\big\}$ the projection on the last variable. The last contribution comes from the fact that, since we are slightly moving $\partial^*E \cap Q^N$ between the levels $\sigma^-$ and $\sigma^+$, we have to take into account that the density is changing. We set then finally
\begin{align*}
T_4^-:=\partial^* E \cap \big(Q \times (\sigma^-,\sigma^+)\big)\,, && 
T_4^+:=\big\{(x',x_N+\delta):(x',x_N)\in T_4^-\big\}\,.
\end{align*}
By construction, we can write
\[
\partial^*F \subseteq \Big(\partial^*E \setminus (T_3^- \cup T_4^-)\Big)\cup \Big(T_1^+\cup T_2^+\cup T_3^+\cup T_4^+\Big)\,.
\]
Thus, since the sets $T_i^+$ are $\H^{N-1}$-essentially pairwise disjoint, and so are also $T_3^-$ and $T_4^-$, one can estimate
\begin{equation}\label{Per}\begin{split}
P_f(F)&-P_f(E) \leq \H^{N-1}_f (T_1^+) + \H^{N-1}_f (T_2^+)\\
&+ \Big(\H^{N-1}_f (T_3^+)- \H^{N-1}_f (T_3^-)\Big)+\Big(\H^{N-1}_f (T_4^+)-\H^{N-1}_f (T_4^-)\Big)\,.
\end{split}\end{equation}
We now estimate the terms in the right hand side of the above inequality one by one: while the first two terms ($\H^{N-1}(T_i^+)$ for $i=1,\,2$) are small by the sole essential boundedness of $f$, to show that the last two terms ($\H^{N-1}(T_i^+)-\H^{N-1}(T_i^-)$ for $i=3,\,4$) are small one needs to use the essential $\alpha$-H\"older assumption on $f$.\par

Let us begin by considering $T_1^+$: by the definition, and also recalling~(\ref{noticethat}), it is easy to show the inclusion
\[
T_1^+\subseteq \Big\{(x',x_N)\in \partial Q\times \big(-a/2,a/2\big):\, \exists\, (x',t)\in \partial^* E,\, x_N-\delta\leq t\leq x_N\Big\}\,,
\]
of course to be intended in the $\H^{N-1}$-sense. Therefore, by using~\eqref{bordo}, we directly find
\begin{equation}\label{est:1}
\H^{N-1}_f(T_1^+)\leq M\delta\H^{N-2}\Big(\partial^*E\cap\big(\partial Q\times(-a/2,a/2)\big)\Big)
\leq 2^{N+1} M\delta a^{N-2}\,.
\end{equation}
Concerning $T_2^+$, it is sufficient to recall~\eqref{P-} in order to obtain
\begin{equation}\label{est:2}
\H^{N-1}_f\big(T_2^+\big)\leq M\delta \H^{N-2}(\partial^*E^{\sigma_-})
\leq 3M\rho a^{N-2}\delta\,.
\end{equation}
We compare now $T_3^+$ and $T_3^-$. Since the projection $\pi'$ is $1$-Lipschitz and $f$ is $\alpha$-H\"older on $Q^N$, we have by~(\ref{PV+}) and recalling that $\bar\delta/(4M^2)\leq \delta\leq \bar\delta$ that
\begin{equation}\label{est:3}
\begin{split}
\H^{N-1}_f(T_3^+)-\H^{N-1}_f(T_3^-)&=\int_{T_3^+}f(x)\, d\H^{N-1}(x)-\int_{T_3^-}f(x)\, d\H^{N-1}(x)\\
&\leq \int_{T_3^-} \Big(f(\pi'(y))-f(y)\Big)\,d\H^{N-1}(y)
\leq M\delta^\alpha\H^{N-1}\big(T_3^-\big)\\
&\leq M\delta^\alpha\H^{N-1}\big(\partial^*E \cap S_{\bar i}\big)
\leq 24 M^3\delta^{\alpha+1} \rho a^{N-2}\,.
\end{split}
\end{equation}
Finally, using again the $\alpha$-H\"older property, we can compare $T_4^+$ and $T_4^-$ as follows
\begin{equation}\label{est:4}
\begin{split}
\H^{N-1}_f(T_4^+)-\H^{N-1}_f(T_4^-)&=\int_{T_4^+}f(x)\,d\H^{N-1}(x)-
\int_{T_4^-}f(x)\, d\H^{N-1}(x)\\
&= \int_{T_4^-} \Big(f(x',x_N+\delta))-f(x',x_N)\Big)\,d\H^{N-1}(x)\\
&\leq M\delta^\alpha\H^{N-1}(T_4^-)
\leq M\delta^{\alpha}a^{N-1}(1+2\rho)\,.
\end{split}
\end{equation}
Plugging~\eqref{est:1}, \eqref{est:2}, \eqref{est:3} and~\eqref{est:4} into~\eqref{Per}, and recalling that $\rho,\, a$ and $\delta$ are as small as we desire, we conclude
\begin{equation}\label{perest}
P_f(F)-P_f(E)\leq 2 M\delta^{\alpha}a^{N-1}\,.
\end{equation}

\step{(v)}{Conclusion of the $\eps-\eps^\alpha$ property.}
We can now conclude very quickly the $\eps-\eps^\alpha$ property. Indeed, take a very small $\eps>0$, and let $\bar\delta=2M\eps/a^{N-1}$. Then, observe that the volume of the set $F=F(\delta)$ is a continuous function of $\delta \in \big(\bar\delta/(4M^2), \bar\delta\big)$: thus, thanks to~(\ref{lowest}) and~(\ref{uppest}), there exists some admissible $\delta$ for which $|F|_f -|E|_f = \eps$. In particular, $\delta$ satisfies
\begin{equation}\label{epsdelta}
\frac{\eps}{2M} \leq \delta a^{N-1} \leq 2M \eps \,.
\end{equation}
Therefore, (\ref{perest}) immediately implies $P_f(F)-P_f(E)\leq (2M)^{1+\alpha} a^{(N-1)(1-\alpha)} \eps^\alpha$.\par

To finish the proof of the $\eps-\eps^\alpha$ property, we have to consider the case when $\eps<0$ and $|\eps| \ll 1$. To do so, still assuming for simplicity that the origin of $\R^N$ belongs to $U \cap \partial^* E$ and the outer normal of $E$ at the origin is the vertical direction, we define $E'= B(1)\setminus E$. Of course $E'$ is a set of finite perimeter, and $\partial^* E= \partial^* E'$ inside the unit ball. We can then apply all the preceding construction to the set $E'$, finding a new set $F'$ such that
\begin{align*}
F' \setminus Q^N = E' \setminus Q^N\,, && \big|F'\big|_f=\big| E'\big|_f + |\eps| \,, && P_f (F') \leq P_f(E') + C |\eps|^\alpha\,,
\end{align*}
being
\[
C= (2M)^{1+\alpha} a^{(N-1)(1-\alpha)}\,.
\]
Defining then $F=\big(E \setminus Q^N\big) \cup \big( Q^N \setminus F'\big)$, we clearly have
\begin{align*}
F \setminus Q^N = E \setminus Q^N\,, && \big|F\big|_f=\big| E\big|_f + \eps \,, && P_f (F) \leq P_f(E) + C |\eps|^\alpha\,,
\end{align*}
so the $\eps-\eps^\alpha$ property is finally established.

\part{II}{The $\eps-\eps^\beta$ property.}
This second part of the proof is devoted to show the $\eps-\eps^\beta$ property for $E$, where $\beta=\beta(\alpha,N)$ is defined in~(\ref{defbeta}). Notice that we can assume $0\leq \alpha<1$, since otherwise $\beta=\alpha$ and then the property has been already shown in Part~I.\par

Our idea is to follow exactly the construction of Part~I except for a single, yet fundamental, detail. To explain this, recall that in Part~I we have selected a cube $Q^N$, of small but fixed side $a$; then, for any small constant $\eps>0$, we have defined $F$ by moving up the set $E$ \emph{in the whole cube $Q^N=Q \times (-a/2,a/2)$} of a distance $\delta \approx \eps$ --in the sense of~(\ref{epsdelta}). What we will do now, instead, will be the following: for every small constant $\eps>0$, we will find a smaller $(N-1)$-dimensional horizontal cube $Q_\eps\subseteq Q$ of side $a\eps^\gamma$, being $\gamma$ a suitable constant to be specified later. Then, we will define $F$ by moving up the set $E$ \emph{only inside $Q_\eps\times (-a/2,a/2)$}, and of a bigger distance $\delta \approx \eps^{1-(N-1)\gamma}$. Of course, this can make sense for arbitrarily small $\eps$ only if $0 < \gamma < \frac 1{N-1}$. Once had this idea, the proof is only a quite simple modification of the argument of Part~I; basically, one only has to select carefully the small $(N-1)$-dimensional cube $Q_\eps$, write down the new form of all the estimates already found in Part~I, and then select the right constant $\gamma$. We will again split the proof in some steps. First of all, we fix the constant $a>0$ and the cube $Q^N=(-a/2,a/2)^N$ exactly as in Step~(i) of Part~I, and we also let $0<\gamma< \frac 1{N-1}$ be a constant, which will be explicitely chosen later.

\step{(i)}{Selection of $H$ ``candidate cubes'' satisfying~(\ref{newbordo}).}
Let $0< \eps \ll 1$ be a suitably small positive number. We aim to select a $(N-1)$-dimensional cube $Q_\eps\subseteq Q$ of side $a\eps^\gamma$; to do so, we will proceed in two different steps. In this first one, we select a high number of cubes satisfying the new version of the boundary estimate~(\ref{bordo}), namely, the estimate~(\ref{newbordo}) below; then, in next step, we will choose one of those cubes, which will fulfill also all the other conditions that we need.\par

We start setting
\[
H:= \left[ \frac{\eps^{(1-N)\gamma}}{2^{N+1}}\right]\,,
\]
and selecting $2H$ disjoint cubes $\big\{\widetilde Q_j\big\}_{j=1,\, \dots ,\, 2H}$ contained in $Q$ and having side $2a\eps^\gamma$; this is of course possible by definition of $H$ as soon as $\eps$ is small enough. Let us now concentrate on a single cube $\widetilde Q_j$, which is centered at some $\bar x'\in Q$, and call $\widetilde Q_j^{1/2}$ the cube centered at $\bar x'$ and having side $a\eps^\gamma$. For every $x'\in \widetilde Q_j^{1/2}$, we call
\[
Q(x'):= \prod_{i=1}^{N-1} \bigg( x'_i - \frac{a\eps^\gamma}2,\, x'_i + \frac{a\eps^\gamma}2\bigg)
\]
the cube of side $a\eps^\gamma$ centered at $x'$, which is of course contained in $\widetilde Q_j$. A simple rough estimate ensures that
\[\begin{split}
\H^{N-1} \Big( \partial^* E &\cap \big( \widetilde Q_j \times (-a/2, a/2)\big)\Big)\\
&\geq \frac{1}{2(N-1)(a\eps^\gamma)^{N-2}} \, \int_{\widetilde Q_j^{1/2}} \H^{N-2} \Big( \partial^* E \cap \big( \partial Q(x')\times (-a/2,a/2)\big)\Big) \, d\H^{N-1}(x')\\
&= \frac{a\eps^\gamma}{2(N-1)} \, \intmed_{\widetilde Q_j^{1/2}} \H^{N-2} \Big( \partial^* E \cap \big( \partial Q(x')\times (-a/2,a/2)\big)\Big) \, d\H^{N-1}(x')\,.
\end{split}\]
As a consequence, there exists some $x_j'\in \widetilde Q_j^{1/2}$ such that
\begin{equation}\label{prevest}
\H^{N-2} \Big( \partial^* E \cap \big( \partial Q(x'_j)\times (-a/2,a/2)\big)\Big) \leq
\frac{2(N-1)}{a\eps^\gamma}\,\H^{N-1} \Big( \partial^* E \cap \big( \widetilde Q_j \times (-a/2, a/2)\big)\Big)\,.
\end{equation}
Observe now that, since the cubes $\widetilde Q_j$ are disjoint and contained in $Q$, it is
\[
\sum_{j=1}^{2H} \H^{N-1} \Big( \partial^* E \cap \big( \widetilde Q_j \times (-a/2, a/2)\big)\Big)
\leq \H^{N-1} \big( \partial^*E \cap Q^N \big) \leq \big( 1+2\rho\big) a^{N-1}\,.
\]
Among the $2H$ cubes $\widetilde Q_j$, there are then at least $H$ cubes which satisfy
\begin{equation}\label{toinsert}
\H^{N-1} \Big( \partial^* E \cap \big( \widetilde Q_j \times (-a/2, a/2)\big)\Big) \leq \frac{\big(1+2\rho\big) a^{N-1}}{H}
\leq 2^{N+1}\big(1+3\rho\big) \big(a\eps^\gamma\big)^{N-1}\,.
\end{equation}
Up to renumbering, we can assume that those ``good'' cubes correspond to the indices $j=1,\, 2,\, \dots \, , \, H$. Hence, for any such $j$ we define $Q_j:= Q(x_j')$. Summarizing, we have found $H$ disjoint $(N-1)$-dimensional cubes $\big\{ Q_j \big\}_{j=1,\, \dots ,\, H}$ of side $a\eps^\gamma$ contained inside $Q$, and inserting~(\ref{toinsert}) in~(\ref{prevest}) we find that each of these cubes satisfies the estimate
\begin{equation}\label{newbordo}
\H^{N-2}\Big(\partial^*E\cap\big(\partial Q_j\times (-a/2,a/2)\big)\Big) \leq 2^{N+2} N \big(a \eps^\gamma\big)^{N-2}\,,
\end{equation}
which can be seen as the new version of~(\ref{bordo}). Exactly as in~(\ref{noticethat}), Vol'pert Theorem~\ref{volpert} allows us to assume, without loss of generality, that for every $1\leq j\leq H$ it is
\begin{equation}\label{New:noticethat}
\partial^*E \cap\big(\partial Q_j\times (-a/2,a/2)\big)=\partial^*\Big(E\cap\big(\partial Q_j\times (-a/2,a/2)\big)\Big)\qquad \H^{N-2}-a.e.\,.
\end{equation}

\step{(ii)}{Choice of the cube $Q_\eps$.}
In this step, we will select one of the $H$ cubes found in Step~(i), and we will call it $Q_\eps$. We will denote by $A_\eps,\, B_\eps,\, G_\eps$ and $\Gamma_\eps$ the intersections of $A,\, B,\, G$ and $\Gamma$ with $Q_\eps$, where we consider the decomposition $Q=A\cup B\cup G \cup \Gamma$ already presented in Step~(ii) of Part~I, and we will write for brevity $Q_\eps^N=Q_\eps \times (-a/2,a/2)$. We aim to choose $Q_\eps$ in such a way that the following holds:
\begin{gather}
\frac{\H^{N-1}\big(\partial^*E \cap Q_\eps^N\big)}{\big(a\eps^\gamma\big)^{N-1}} \leq 1+2^{N+4}\cdot 8\rho\,,\label{NewP2}\\
\frac{\H^{N-1}\big(\partial^*E \cap Q_\eps^N\setminus\{-a\rho<x_N<a\rho\}\big)}{\big(a\eps^\gamma\big)^{N-1}} \leq 2^{N+4} \rho\,,\label{NewP3}\\
\frac{\H^N\big(E \cap Q_\eps^N\cap\{x_N>0\}\big)}{a^N\eps^{\gamma(N-1)}/2} \leq 2^{N+4} \rho\,,\label{NewV2}\\
\frac{\H^{N-1}\big(G_\eps\big)}{\big(a\eps^\gamma\big)^{N-1}}\geq 1 - 2^{N+4}\cdot 6 \rho \,.\label{newlast}
\end{gather}
Let us show that this is possible. First of all, writing for brevity $Q_j^N= Q_j\times (-a/2,a/2)$ for every $1\leq j \leq H$, we can apply~(\ref{P3}) to find
\[
a^{N-1} \rho \geq \H^{N-1}\big(\partial^*E \cap Q^N\setminus\{-a\rho<x_N<a\rho\}\big)
\geq \sum_{j=1}^H \H^{N-1}\big(\partial^*E \cap Q_j^N\setminus\{-a\rho<x_N<a\rho\}\big)\,.
\]
As a consequence, \emph{strictly more than $75\%$} of the $H$ cubes satisfy
\[
\H^{N-1}\Big(\partial^*E \cap Q_j^N\setminus\{-a\rho<x_N<a\rho\}\Big) \leq 4\, \frac{a^{N-1}\rho}{H}
\leq 2^{N+4} \rho \big(a\eps^\gamma\big)^{N-1} \,,
\]
that is, (\ref{NewP3}). In the very same way, applying~(\ref{V2}) we observe that more than $75\%$ of the cubes satisfy~(\ref{NewV2}), and applying~(\ref{est:G}) we observe than more than $75\%$ of the cubes satisfy~(\ref{newlast}).\par
Some additional care is required to obtain also~(\ref{NewP2}). In fact, let $\pi: Q^N\to Q$ be the projection on the horizontal variables, and define the measure $\mu \in \mathcal M(Q)$ as
\[
\mu := \pi_\# \Big( \H^{N-1} \res \big(\partial^* E \cap Q^N\big)\Big) - \H^{N-1}\res G\,.
\]
In other words, for every $(N-1)$-dimensional Borel set $V\subseteq Q$, we set
\[
\mu (V) := \H^{N-1} \Big( \partial^* E \cap \big\{ (x',x_N)\in Q^N:\, x'\in V\big\}\Big)- \H^{N-1}\big( V \cap G\big)\,.
\]
By construction and by definition of $G$, one clearly has that $\mu$ is a positive measure; hence, by~(\ref{P2}), (\ref{P3}) and~(\ref{est:G}) we deduce
\[
\| \mu \| = \mu (Q) = \H^{N-1}\big( \partial^* E \cap Q^N\big)-\H^{N-1}(G)
\leq 8 \rho a^{N-1}\,.
\]
Thus, arguing as before, we find that more than $75\%$ of the cubes satisfy
\[
\mu\big(Q_j\big) \leq 4\, \frac{8\rho a^{N-1}}{H} \leq 2^{N+4}\cdot 8 \rho \big(a\eps^\gamma\big)^{N-1}\,.
\]
For each of those cubes, it is clearly
\[
\H^{N-1}\big(\partial^*E \cap Q_j^N\big) = \H^{N-1} \big( Q_j \cap G\big) + \mu\big( Q_j\big)
\leq \big( 1+ 2^{N+4}\cdot 8 \rho\big) \big(a \eps^\gamma\big)^{N-1}\,,
\]
thus~(\ref{NewP2}) holds.\par
As a consequence, there must be at least one of the cubes $Q_j$ which satisfies contemporarily~(\ref{NewP2}), (\ref{NewP3}), (\ref{NewV2}) and~(\ref{newlast}), hence we conclude this step by calling $Q_\eps$ one of those ``good'' cubes. Recall that, since $Q_\eps$ is one of the $H$ cubes found in Step~(i), then also~(\ref{newbordo}) holds true.

\step{(iii)}{Definition of $F$ and evaluation of its volume and perimeter.}
We can now easily give the definition of $\sigma^+,\,\sigma^-$ and $F$, and evaluate the volume and perimeter of $F$, performing the very same arguments done in Steps~(iii) and~(iv) of Part~I. In fact, we will see that~(\ref{newbordo})--(\ref{newlast}) are the analogous of everything that we really needed there.\par

We start again by fixing some $\bar\delta\ll a$: then, exactly as we proved~(\ref{PV+}), we can use~(\ref{NewP3}) and~(\ref{NewV2}) to find $a\rho<\sigma^+< a/2 - 2\rho$ such that the horizontal strip $S_{\bar i}=Q_\eps \times (\sigma^+, \sigma^++\bar\delta)$ satisfies
\begin{equation}\label{NewPV+}
a\H^{N-1}(\partial^*E\cap S_{\overline i})+\H^N(E\cap S_{\overline i})
\leq 2^{N+6} \bar\delta \rho (a\eps^\gamma)^{N-1}\,.
\end{equation}
Moreover, exactly as we used~(\ref{P3}) to prove~(\ref{P-}), we can use~(\ref{NewP3}) to get the existence of some $-a/2 < \sigma^-< -a\rho$ such that
\begin{equation}\label{NewP-}
\H^{N-2}\big(\partial^*(E\cap Q_\eps^N)^{\sigma^-}\big)\leq 3\cdot 2^{N+4}\rho \, \frac{\big(a \eps^\gamma\big)^{N-1}}a
\leq 3\cdot 2^{N+4}\rho \, \big(a \eps^\gamma\big)^{N-2}\,.
\end{equation}
We can now generalize also the definition of the competitor set in the obvious way. More precisely, for every $\bar\delta/(4M^2)\leq\delta\leq \bar\delta$ we set $F=F(\delta)$ as
\[
x\in F \Longleftrightarrow \left\{\begin{array}{ll}
x\in E \setminus Q_\eps^N \,,\\
x \in E\cap Q_\eps^N\cap \big(\{x_N\leq \sigma^-\} \cup \{x_N> \sigma^+ + \delta\big\})\,,\\
(x',\sigma^-)\in E\cap Q_\eps^N\ \mbox{and}\ \sigma^-< x_N\leq \sigma^-+\delta\,, \\
(x',x_N-\delta)\in E\cap Q_\eps^N\ \mbox{and}\ \sigma^-+\delta<x_N \leq \sigma^+ +\delta\,.
\end{array}\right.
\]
Let us evaluate now the volume and perimeter of $F$. Concerning the volume, similarly as in Part~I we define
\begin{align*}
E^+:= F\setminus E\,, && E^-_1:= E\setminus F \cap \big(Q_\eps\times (\sigma^-,\,\sigma^+)\big)\,, 
&& E^-_2:= E\setminus F \cap \big(Q_\eps\times (\sigma^+,\,\sigma^++\delta)\big) \,,
\end{align*}
so that $F=E \cup E^+ \setminus \big( E^-_1\cup E^-_2\big)$ and~(\ref{est:Q}) holds, namely, $|F|_f-|E|_f=|E^+|_f-|E^-_1|_f-|E_2^-|_f$. We start with the estimate of the volume of $E^-_2$: recalling that $f\leq M$ in $Q^N$, and using~(\ref{NewPV+}), we get
\begin{equation}\label{New:E2-}
\big|E^-_2\big|_f
\leq\big|E\cap S_{\bar i}\big|_f
\leq 2^{N+6} M \bar\delta \rho (a\eps^\gamma)^{N-1}
\leq 2^{N+8} M^3 \delta \rho (a\eps^\gamma)^{N-1}\,.
\end{equation}
To estimate from below the volume $E^+$, it is enough to recall that $E^+\supseteq \cup_{x'\in G_\eps} E^+_{x'}\,$, and that $E^+_{x'}$ is a segment of length $\delta$ for every $x'\in G_\eps$. Thus, by~(\ref{newlast}) we get
\begin{equation}\label{New:E+}
|E^+|_f\geq \frac{1}{M}\, \delta \H^{N-1}(G_\eps)\geq \frac{1}{M}\, \delta \big(a\eps^\gamma\big)^{N-1} \big(1 - 2^{N+4}\cdot 6 \rho\big)\,,
\end{equation}
and conversely
\begin{equation}\label{New:E+G}
\big|E^+\cap (G_\eps\times(-a/2,a/2))\big|_f \leq M \delta\H^{N-1}(G_\eps)\leq M\delta \big(a \eps^\gamma\big)^{N-1}\,.
\end{equation}
Now, since $A_\eps,\, B_\eps,\, G_\eps$ and $\Gamma_{3,\eps}$ are contained by definition in $A,\,B,\, G$ and $\Gamma_3$, by~(\ref{E+B}) and~(\ref{VE3}) we immediately deduce that
\begin{equation}\label{New:E+B+VE3}
\Big| \Big(E^-_1 \cup E^+\Big) \cap \Big((A_\eps\cup B_\eps \cup \Gamma_{3,\eps}) \times (-a/2,a/2)\Big)\Big|_f
=\Big|E_1^- \cap \big( G_\eps \times (-a/2,a/2)\big)\Big|_f=0\,.
\end{equation}
Finally observe that, thanks to~(\ref{NewP2}) and~(\ref{newlast}) and by the definition of $G$, we have
\begin{equation}\label{New:est-gamma}\begin{split}
\H^{N-1}\Big(\partial^*E &\cap \big( A_\eps \cup B_\eps \cup \Gamma_\eps\big)\times (-a/2,a/2)\Big)\\
&=\H^{N-1}\Big(\partial^*E \cap Q_\eps^N\Big) - \H^{N-1}\Big(\partial^*E \cap \big(G_\eps\times (-a/2,a/2)\big)\Big)\\
&\leq 2^{N+4}\cdot 14 \rho \big(a \eps^\gamma\big)^{N-1}\,,
\end{split}\end{equation}
which is the perfect analogous of~(\ref{est-gamma}). Thus, exactly as in~(\ref{VE2}), we obtain
\begin{equation}\label{New:VE2}
\Big|\big( E_1^- \cup E^+\big) \cap \Big(\big(\Gamma_{0,\eps} \cup \Gamma_{1,\eps} \cup \Gamma_{2,\eps}\big)\times (-a/2,a/2)\Big)\Big|_f
\leq 2^{N+4}\cdot 14 M\rho \big(a \eps^\gamma\big)^{N-1}\delta \,.
\end{equation}
We can finally write down the estimates for $|F|_f-|E|_f$. Indeed, on one hand, by~\eqref{New:E+}, \eqref{New:VE2}, \eqref{New:E+B+VE3} and~(\ref{New:E2-}), and up to take $\rho$ sufficiently small, we get
\[
|F|_f-|E|_f\geq 
\delta \big(a\eps^\gamma\big)^{N-1} \bigg( \frac{1}{M}\, \big(1 - 2^{N+4}\cdot 6 \rho\big)
-2^{N+4}\cdot 14 M\rho -2^{N+8} M^3 \rho \bigg)
\geq \frac{\delta \big(a\eps^\gamma\big)^{N-1}}{2M}\,.
\]
On the other hand, putting together~\eqref{New:E+G}, \eqref{New:E+B+VE3}, and~\eqref{New:VE2}, we also find
\[
|F|_f-|E|_f\leq |E^+|_f\leq M\delta \big(a \eps^\gamma\big)^{N-1} \Big( 1 +2^{N+4}\cdot 14 \rho\Big)
\leq 2M\delta\big(a \eps^\gamma\big)^{N-1}\,.
\]
The same argument used in Step~(v) of Part~I ensures then that, if we define $\displaystyle\bar\delta=\frac{2M\eps^{1-(N-1)\gamma}}{a^{N-1}}$, then there exists an admissible $\delta$ such that
\begin{align}\label{tte}
|F|_f-|E|_f = \eps\,, &&
\frac{\eps^{1-(N-1)\gamma}}{2M a^{N-1}} \leq \delta \leq \frac{2M\eps^{1-(N-1)\gamma}}{a^{N-1}}\,.
\end{align}
Let us now pass to study the perimeter of $F$. Exactly as in Step~(iv) of Step~I, we define
\begin{align*}
T_1^+&:=\big(\partial^* F \setminus \partial^* E\big)\cap \big(\partial Q_\eps \times (-a/2,a/2)\big)\,, 
& T_2^+&:=\partial^*\big(E \cap Q_\eps^N\big)^{\sigma^-} \times (\sigma^-, \sigma^-+\delta)\,,\\
T_3^-&:=\partial^*E \cap \big(Q_\eps \times (\sigma^+,\sigma^+ + \delta)\big)\,, & T_3^+&:=\pi'\big(T_3^-\big)\,,\\
T_4^-&:=\partial^* E \cap \big(Q_\eps \times (\sigma^-,\sigma^+)\big)\,, &
T_4^+&:=\big\{(x',x_N+\delta):(x',x_N)\in T_4^-\big\}\,.
\end{align*}
where $\pi': Q_\eps^N \rightarrow Q_\eps\times \big\{x_N=\sigma^+ + \delta\big\}$ is the projection on the last variable. Then, we clearly still have the validity of~(\ref{Per}), so we need to estimate the $\H^{N-1}_f$ measures of the sets $T_i^\pm$.\par
The same argument which proved~(\ref{est:1}), keeping in mind~(\ref{New:noticethat}) and using~(\ref{newbordo}) in place of~(\ref{bordo}), gives now
\begin{equation}\label{New:est:1}
\H^{N-1}_f(T_1^+)\leq M\delta\H^{N-2}\Big(\partial^*E\cap\big(\partial Q_\eps\times(-a/2,a/2)\big)\Big)
\leq 2^{N+2} N M \big(a \eps^\gamma\big)^{N-2}\delta \,.
\end{equation}
Concerning $T_2^+$, (\ref{NewP-}) immediately gives
\begin{equation}\label{New:est:2}
\H^{N-1}_f\big(T_2^+\big)\leq M\delta \H^{N-2}\big(\partial^*\big(E \cap Q_\eps^N\big)^{\sigma^-}\big)
\leq 3\cdot 2^{N+4} M \rho \, \big(a \eps^\gamma\big)^{N-2}\delta\,.
\end{equation}
Exactly as in~(\ref{est:3}), a comparison between the $\H^{N-1}(T^+_3)$ and $\H^{N-1}(T^-_3)$ readily comes from~(\ref{NewPV+}) and using the $\alpha$-H\"older property of $f$, since
\begin{equation}\label{New:est:3}
\begin{split}
\H^{N-1}_f(T_3^+)-\H^{N-1}_f(T_3^-)&\leq \int_{T_3^-} \Big(f(\pi'(y))-f(y)\Big)\,d\H^{N-1}(y)
\leq M\delta^\alpha\H^{N-1}\big(T_3^-\big)\\
&\leq M\delta^\alpha\H^{N-1}\big(\partial^*E \cap S_{\bar i}\big)
\leq \frac{2^{N+8} M^3 \delta^{\alpha+1} \rho (a\eps^\gamma)^{N-1}}a\,.
\end{split}
\end{equation}
Finally, using once again the $\alpha$-H\"older property of $f$ as in~(\ref{est:4}), and recalling~(\ref{NewP2}), we get
\begin{equation}\label{New:est:4}
\begin{split}
\H^{N-1}_f(T_4^+)-\H^{N-1}_f(T_4^-)&= \int_{T_4^-} \Big(f(x',x_N+\delta))-f(x',x_N)\Big)\,d\H^{N-1}(x)\\
&\leq M\delta^\alpha\H^{N-1}(T_4^-)
\leq M\delta^{\alpha}\big(a\eps^\gamma\big)^{N-1} \big(1+2^{N+4}\cdot 8\rho\big)\,.
\end{split}
\end{equation}

\step{(iv)}{Choice of $\gamma$ and conclusion.}
We are finally ready to conclude our proof. In the preceding steps, we have shown that for every $0<\eps \ll 1$ there exists some set $F$ which equals $E$ out of a small cube $Q^N$ and such that $|F|_f - |E|_f=\eps$. Moreover, putting together~(\ref{New:est:1}), (\ref{New:est:2}), (\ref{New:est:3}) and~(\ref{New:est:4}), and recalling~(\ref{tte}), we also have the estimate
\begin{equation}\label{lastestimate}
P_f(F)-P_f(E)\leq C' \Big(\delta \eps^{\gamma(N-2)} + \delta^\alpha \eps^{\gamma(N-1)}\Big)
\leq C'' \Big( \eps^{1-\gamma}+\eps^{\alpha+\gamma(N-1)(1-\alpha)}\Big)\,,
\end{equation}
being
\[
C'' = \frac{2^{N+4} N M^2}a + (2M)^{1-\alpha} a^{(N-1)(1-\alpha)}\,.
\]
Recall now that $0<\gamma<\frac{1}{N-1}$ is a fixed constant, still to be chosen. It is then finally clear what is the best choice for $\gamma$: indeed, notice that $\gamma_1= 1-\gamma$ is a decreasing function of $\gamma$, while $\gamma_2=\alpha+\gamma(N-1)(1-\alpha)$ is an increasing one, and notice also that, since $0<\alpha<1$, then $\gamma_1>\gamma_2$ (resp., $\gamma_1<\gamma_2$) for $\gamma\approx 0$ (resp., $\gamma\approx \frac 1{N-1}$). Therefore, the optimal choice of $\gamma$ corresponds to the situation when $\gamma_1=\gamma_2$, which means that we can decide
\begin{equation}\label{defgamma}
\gamma:= \frac{1-\alpha}{\alpha+N(1-\alpha)}\,.
\end{equation}
Summarizing, we have been able to build a set $F$ with $|F|_f-|E|_f=\eps$ and
\[
P_f(F)-P_f(E) \leq 2C'' \eps^\beta
\]
with
\begin{equation}\label{newdefbeta}
\beta = \frac{\alpha+(N-1)(1-\alpha)}{\alpha+N(1-\alpha)}\,,
\end{equation}
which corresponds to~(\ref{defbeta}). To conclude the $\eps-\eps^\beta$ property, we only have to deal with the case when $\eps<0$ and $|\eps|\ll 1$; however, the case of negative $\eps$ can be derived from the case of positive $\eps$ exactly as we did in Step~(v) of Part~I. Hence, also this second part is concluded.

\part{III}{The case when $f$ is continuous.}
Let us conclude our proof by considering the case when $f$ is only essentially bounded and continuous. By the result of Part~II, the essential boundedness of $f$, thus the essential $\alpha$-H\"older property with $\alpha=0$, already tells us that the $\eps-\eps^{\frac{N-1}N}$ property holds true with some constant $C''$, since $\beta=\frac{N-1}N$ when $\alpha=0$ by~(\ref{newdefbeta}). What we want to do, is to show the validity of the same property with \emph{any} arbitrarily small constant.\par
To do so, recall that the estimate of $P_f(F)-P_f(E)$ comes from four terms, see~(\ref{New:est:1})--(\ref{New:est:4}) above. Our strategy will be to slightly modify the definition of $\bar\delta$ in order to decrease as desired the first two terms; unfortunately, while doing so the last two terms will correspondingly increase. However, using the fact that $f$ is continuous, instead of only essentially bounded, we will be able to let also the last two terms become arbitrarily small.\par

Let us be more precise: we fix a small number $c>0$ and we aim to show the $\eps-\eps^{\frac{N-1}N}$ property with constant $C=c$. To do so, we recall the construction and the estimates of Step~II with $\gamma=1/N$, which corresponds to the case $\alpha=0$. The only difference now, is that we fix a large constant $L$, to be specified later, and we want to build a set $F$ such that
\[
|F|_f -|E|_f = \tilde \eps := \frac \eps L\,.
\]
To do so, the only required change is to define
\[
\bar\delta := \frac{2M\eps^{1-(N-1)\gamma}}{L a^{N-1}}
=\frac{2M\eps^{\frac 1N}}{L a^{N-1}}\,,
\]
which of course reduces to the old choice of $\bar\delta$ if $L=1$. Then, we find again some $\delta$ satisfying
\begin{align}\label{New:tte}
|F|_f-|E|_f = \tilde\eps\,, &&
\frac{\tilde\eps^{\frac 1N}}{2ML^{\frac{N-1}N} a^{N-1}} \leq \delta \leq \frac{2M\tilde\eps^{\frac 1N}}{L^{\frac{N-1}N}a^{N-1}}\,.
\end{align}
which still reduces to~(\ref{tte}) if $L=1$. By~(\ref{New:est:1}) and~(\ref{New:est:2}), recalling~(\ref{New:tte}) and up to take $\rho$ small enough, we know that
\begin{equation}\label{nomestrano1}\begin{split}
\H^{N-1}_f(T_1^+)+ \H^{N-1}_f(T_2^+) &\leq 2^{N+3} N M \big(a \eps^\gamma\big)^{N-2}\delta
=2^{N+3} N M a^{N-2}\tilde\eps^{\frac{N-2}N} L^{\frac{N-2}N} \delta\\
&\leq \frac{2^{N+4} N M^2}{aL^{\frac 1N}}\, \tilde\eps^{\frac{N-1}N}
\leq \frac c 3\, \tilde\eps^{\frac{N-1}N}\,,
\end{split}\end{equation}
where the last inequality holds true up to choose a sufficiently large constant $L$.\par
We have now to evaluate $\H^{N-1}_f(T_3^+)-\H^{N-1}_f(T_3^-)$ and $\H^{N-1}_f(T_4^+)-\H^{N-1}_f(T_4^-)$. If we just insert the new choice of $\delta$ into~(\ref{New:est:3}) and~(\ref{New:est:4}), these two estimates become worse because of the presence of the big constant $L$: in fact, it is now time to use the continuity assumption on $f$. Let us then call $\omega$ the standard continuity modulus of $f$ on $Q^N$, i.e.,
\[
\omega(c) = \sup \Big\{ \big|f(x) - f(y)\big|:\, x,\, y\in Q^N,\, |x-y|\leq c\Big\}\,;
\]
since $f$ is continuous, thus uniformly continuous on $Q^N$, we have that $\omega\searrow 0$ when $c \searrow 0$. We can then easily modify the calculation of~(\ref{New:est:3}) as
\begin{equation}\label{nomestrano2}\begin{split}
\H^{N-1}_f(T_3^+)-\H^{N-1}_f(T_3^-)&\leq \int_{T_3^-} \Big(f(\pi'(y))-f(y)\Big)\,d\H^{N-1}(y)
\leq \omega(\delta)\H^{N-1}\big(T_3^-\big)\\
&\leq \omega(\delta) \H^{N-1}\big(\partial^*E \cap S_{\bar i}\big)
\leq \frac{2^{N+7} M \rho}a \,\tilde\eps \omega(\delta) 
\leq \frac c 3\, \tilde\eps^{\frac{N-1}N}\,,
\end{split}\end{equation}
where the last inequality is true if $\tilde\eps$ is small enough. Finally, the estimate~(\ref{New:est:4}) now becomes
\begin{equation}\label{bisnum}\begin{split}
\H^{N-1}_f(T_4^+)-\H^{N-1}_f(T_4^-)&= \int_{T_4^-} \Big(f(x',x_N+\delta))-f(x',x_N)\Big)\,d\H^{N-1}(x)\\
&\leq \omega(\delta) \H^{N-1}(T_4^-)
\leq \omega(\delta) \big(a\eps^\gamma\big)^{N-1} \big(1+2^{N+4}\cdot 8\rho\big)\\
&\leq 2 L^{\frac{N-1}N} a^{N-1} \tilde \eps^{\frac{N-1}N}\, \omega(\delta)
\leq 2 L^{\frac{N-1}N} a^{N-1} \tilde \eps^{\frac{N-1}N}\, \omega\bigg(\frac{2M\tilde\eps^{\frac 1N}}{L^{\frac{N-1}N}a^{N-1}}\bigg)\,,\hspace{-20pt}
\end{split}
\end{equation}
using also~(\ref{New:tte}). We can select a sufficiently small $\bar\eps$ such that, whenever $0<\tilde\eps<\bar\eps$, one has
\[
\omega\bigg(\frac{2M\tilde\eps^{\frac 1N}}{L^{\frac{N-1}N}a^{N-1}}\bigg) \leq \frac{c}{6 L^{\frac{N-1}N} a^{N-1}}
\]
(recall that $L$ has been already fixed). As a consequence, (\ref{bisnum}) yields
\begin{equation}\label{nomestrano3}
\H^{N-1}_f(T_4^+)-\H^{N-1}_f(T_4^-) \leq \frac c 3\, \tilde\eps^{\frac{N-1}N}\,.
\end{equation}
Summarizing, for any $c>0$ we have found $0<\bar\eps\ll 1$ such that, for any $0<\tilde\eps<\bar\eps$, there exists $F$ which, by~(\ref{nomestrano1}), (\ref{nomestrano2}) and~(\ref{nomestrano3}), satisfies
\begin{align*}
|F|_f-|E|_f = \tilde\eps\,, && 
P_f(F) - P_f(E) \leq c \tilde\eps^{\frac{N-1}N}\,.
\end{align*}
Arguing in the usual way to treat the case of $\tilde\eps<0,\, |\tilde\eps|<\bar\eps$, we have then concluded the validity of the $\eps-\eps^{\frac{N-1}N}$ property with any constant $c>0$.
\end{proof}

\begin{remark}
We underline that, in Theorem~\mref{eps-epsbeta}, the validity of the $\eps-\eps^\beta$ property with any small constant is a peculiarity of the case when $f$ is continuous, but it cannot be inferred for the general case of a $\alpha$-H\"older function $f$.
\end{remark}

\section{An example of an unbounded isoperimetric set\label{example}}

This section is devoted to show an example of an essentially bounded but discontinuous function $f$ which admits an unbounded isoperimetric set. This will show the sharpness of the assumption in Theorem~\mref{eps-epsbeta}.\par

Let $\big\{ B_i \big\}_{i\in\N}$ be a sequence of disjoint balls of Euclidean volume $\big|B_i\big|_\eu=1/2^i$, sufficiently far from each other, and let $f:\R^N\to \R$ be defined as
\[
f(x):= \left\{ \begin{array}{ll}
1\qquad & \hbox{if $x\in \bigcup_{\,i} \partial B_i$}\,, \\
M & \hbox{otherwise}\,, \\
\end{array}\right.
\]
where $M>1$ is a constant, big enough, to be precised later. We will show that $B=\cup_i B_i$, which is an unbounded set, is the unique isoperimetric set of volume $M$. To do so, let us pick a smooth set $E\subseteq \R^N$ with $|E|_f=M$; we aim to show that, for a suitable constant $\xi>0$, one has
\begin{equation}\label{toget}
P_f(E) \geq P_f(B) + \xi \eta^{\frac{N-1}N}\,, \qquad\hbox{where}\quad\eta:=\frac{\big| B \triangle E\big|_\eu}2\,,
\end{equation}
being $B \triangle E = (B\setminus E) \cup (E \setminus B)$ the symmetric difference between $E$ and $B$. By the density of smooth sets among sets of finite perimeter, (\ref{toget}) will show that $B$ is the unique isoperimetric set of volume $M$.

\begin{claim}
It is admissible to assume that every connected component of $E$ intersects exactly one of the balls $B_i$.
\end{claim}
\begin{proof}
Assume first that a connected component of $E$ intersects two different balls. We argue exactly as in the proof of Theorem~\mref{boundedness}: one can build a competitor $\widetilde E$ by cutting the part of $E$ which is far enough from each of the balls, and replacing it with a ball of the same volume (not intersecting $B\cup E$); the very same calculation done in the proof of Theorem~\mref{boundedness} ensures that, if the balls are chosen sufficiently far from each other, there will be such a set $\widetilde E$ having perimeter smaller than that of $E$. Since by construction $\big| B\triangle E\big|_f=\big| \widetilde B\triangle E\big|_f$, we can reduce ourselves to show~(\ref{toget}) for the set $\widetilde E$; this shows that it is admissible to assume that no connected component of $E$ intersects two different balls.\par
Suppose now, instead, that there is a connected component of $E$ which does not intersect any ball; then, we can simply translate this component around $\R^N$ until it touches one of the balls $B_i$, or one of the other connected components which in turn touches some ball. Since the density is constant in $\R^N\setminus B$, this translation does not effect the perimeter nor the volume of $E$, hence it is admissible to assume that every connected component of $E$ intersects at least one ball. The proof of the claim is then concluded.
\end{proof}

Thanks to the above claim, for every $i$ we can consider the (possibly empty) connected component of $E$ which intersects $B_i$, and subdivide it into two parts, the set $E_i\subseteq B_i$ and the remaining set $F_i$, as depicted in Figure~\ref{fig:ex}. We call now
\begin{align*}
\eps_i:= \big|B_i\setminus E_i\big|_\eu \,, &&
\delta_i:=\big|F_i\big|_\eu\,,
\end{align*}
and notice that
\begin{equation}\label{reclat}
\sum_i \eps_i = \sum_i \delta_i = \eta
\end{equation}
by the definition of $\eta$ in~(\ref{toget}) and since $|E|_f=|B|_f=M$.
\begin{figure}[htbp]
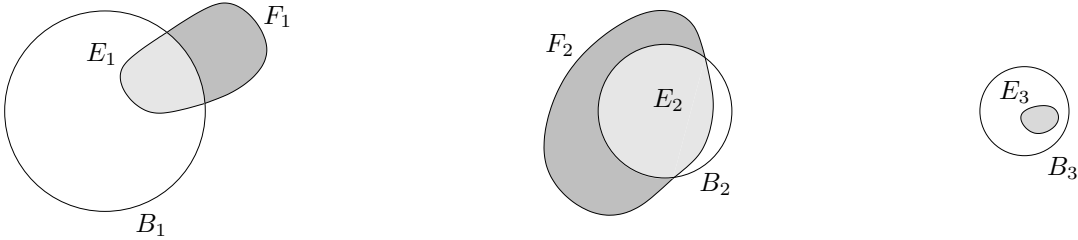
\caption{Sketch of the situation in the Example of Section~\ref{example}; the set $F_3$ is empty.}\label{fig:ex}
\end{figure}
Our next observation is the following.
\begin{claim}\label{cl2}
For every $i\in\N$ one has
\[
P_f(B_i\cup F_i) - P_f(B_i) \geq \frac{M-1}2\, N \omega_N^{1/N} \delta_i^{\frac{N-1}N}\,.
\]
\end{claim}
\begin{proof}
First of all, we know by the Euclidean isoperimetric inequality that
\begin{equation}\label{almlu}
\H^{N-1}\big(\partial F_i\big) = P_\eu(F_i)\geq N \omega_N^{1/N} \big| F_i \big|_\eu^{\frac{N-1}N}
= N \omega_N^{1/N} \delta_i^{\frac{N-1}N}\,.
\end{equation}
Consider now the boundary of $F_i$: it is done by two parts, namely, $\Gamma_1=\partial F_i\setminus \overline B_i$ and $\Gamma_2=\partial F_i\cap\partial B_i$. Call now $\pi:\R^N\setminus \overline B_i\to \partial B_i$ the projection on $B_i$: since the ball in convex, $\pi$ is $1$-Lipschitz; therefore,
\[
\H^{N-1} ( \Gamma_1 ) \geq \H^{N-1} \big( \pi ( \Gamma_1 )\big) \geq \H^{N-1} ( \Gamma_2 )\,.
\]
As a consequence,
\[\begin{split}
P_f(B_i\cup F_i) - P_f(B_i) &= \H^{N-1}_f \big(\Gamma_1\big) - \H^{N-1}_f \big(\Gamma_2\big)
= M \H^{N-1} \big(\Gamma_1\big) - \H^{N-1} \big(\Gamma_2\big)\\
&\geq (M-1) \H^{N-1} \big(\Gamma_1\big)
\geq \frac{M-1}2\, \H^{N-1} \big(\partial F_i\big)\\
&\geq \frac{M-1}2\, N \omega_N^{1/N} \delta_i^{\frac{N-1}N}\,,
\end{split}\]
recalling that $\partial F_i = \Gamma_1\cup \Gamma_2$ and~(\ref{almlu}). The proof of the claim is then concluded.
\end{proof}

As an immediate corollary, just adding over $i\in\N$, using the concavity of $t\mapsto t^{\frac{N-1}N}$, and recalling~(\ref{reclat}), we get
\begin{equation}\label{tocomp}\begin{split}
\sum_{i\in\N} P_f(B_i\cup F_i) - P_f(B)&=
\sum_{i\in\N} \Big(P_f(B_i\cup F_i) -P_f(B_i)\Big)\geq\frac{M-1}2\,N\omega_N^{1/N}\sum_{i\in\N}\delta_i^{\frac{N-1}N}\\
&\geq \frac{M-1}2\, N \omega_N^{1/N} \Big(\sum\nolimits_{i\in\N} \delta_i\Big)^{\frac{N-1}N}
= \frac{M-1}2\, N \omega_N^{1/N} \eta^{\frac{N-1}N}\,.
\end{split}\end{equation}
Now, since a quick observation tells us that
\begin{equation}\label{quickobs}
P_f(E_i\cup F_i)-P_f(B_i\cup F_i) \geq P_f(E_i)-P_f(B_i)\,,
\end{equation}
and thanks to~(\ref{tocomp}), we are basically reduced to evaluate $P_f(E_i)-P_f(B_i)$ in terms of $\eps_i$. We can start with an easy bound.
\begin{claim}\label{cleps}
For every $i\in\N$, one has
\[
P_f(E_i)-P_f(B_i) \geq -N\omega_N^{1/N} \eps_i^{\frac{N-1}N}\,.
\]
\end{claim}
\begin{proof}
The Euclidean isoperimetric inequality tells us that
\[
P_\eu (E_i) \geq N\omega_N^{1/N} |E_i|_\eu^{\frac{N-1}N}
=N\omega_N^{1/N} \Big(|B_i|_\eu-\eps_i\Big)^{\frac{N-1}N}\,;
\]
therefore, again by the concavity of $t\mapsto t^{\frac{N-1}N}$, we have
\[\begin{split}
P_f(E_i)- P_f(B_i) &\geq P_\eu(E_i)-P_\eu(B_i)
\geq N\omega_N^{1/N} \bigg( \Big(|B_i|_\eu-\eps_i\Big)^{\frac{N-1}N} - |B_i|_\eu^{\frac{N-1}N} \bigg)\\
&\geq -N\omega_N^{1/N}\eps_i^{\frac{N-1}N}\,.
\end{split}\]
\end{proof}
Unfortunately, we cannot simply conclude comparing~(\ref{tocomp}) and last claim, because the concavity of $t\mapsto t^{\frac{N-1}N}$ now works against our estimate, since $\sum -\eps_i^{\frac{N-1}N}\leq -\eta^{\frac{N-1}N}$. To overcome this problem, we subdivide $\N$ into two parts, namely,
\begin{align*}
I := \bigg\{ i\in\N:\, \frac{\eps_i}{\big| B_i\big|_\eu} \leq \frac 14 \bigg\} \,, &&
J :=\bigg\{ i\in\N:\, \frac{\eps_i}{\big| B_i\big|_\eu} > \frac 14 \bigg\} \,;
\end{align*}
in words, the indices belonging to $I$ (resp., $J$) are those corresponding to sets $E_i$ which contain more (resp., less) than $75\%$ of the corresponding ball $B_i$. The key point is that for indices in $I$ something much stronger than Claim~\ref{cleps} can be found.
\begin{claim}\label{claimI}
For every $i\in I$, one has $P_f(E_i)\geq P_f(B_i)$.
\end{claim}
\begin{proof}
First of all, we consider the spherical symmetrization $E^*$ of $E_i$, which satisfies $P_f(E^*)\leq P_f(E_i)$ by Theorem~\ref{sphsym} (and since $f$ is obviously radial in a neighborhood of $B_i$); notice that it is also $|E^*|_\eu=|E|_\eu$, since $f$ is constant in $B_i$. Then, we write $\partial E^*=\Gamma_1\cup \Gamma_2$, where $\Gamma_1=\partial E^*\cap B_i$ and $\Gamma_2=\partial E^*\cap \partial B_i$; in particular, $\Gamma_2$ is a (possibly empty) spherical cap in $B_i$. Let us then call $P$ the $(N-1)$-dimensional ball whose boundary coincides with the boundary of $\Gamma_2$ as a subset of $\partial B_i$. In other words, since $\Gamma_2$ is a spherical cap then up to a rotation one has $\Gamma_2 = \big\{ x \in \partial B_i:\, x \cdot \nu \leq \kappa \big\}$ for suitable $\kappa\in\R$ and $\nu\in\S^{N-1}$, we define the ball $P=\big\{ x \in B_i:\, x \cdot \nu = \kappa \big\}$; the situation is depicted in Figure~\ref{fig:cl3}. We have now to distinguish two cases.
\begin{figure}[htbp]
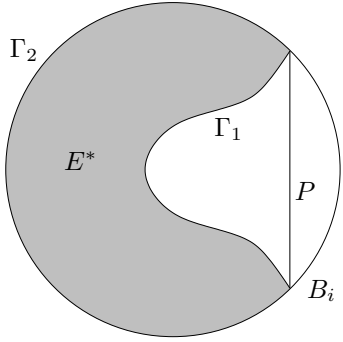
\caption{Situation in Claim~\ref{claimI}.}\label{fig:cl3}
\end{figure}
\case{I}{$\H^{N-1}(\Gamma_2)\geq \H^{N-1}(\partial B_i)/2$.}
In this case, as in Claim~\ref{cl2} we call $\pi$ the projection over $P$, which is $1$-Lipschitz by the convexity of $P$; then, we can again estimate
\[
\H^{N-1}(\Gamma_1) \geq \H^{N-1}\big( \pi(\Gamma_1)\big) \geq \H^{N-1}(P)\,.
\]
Moreover, since $\H^{N-1}(\Gamma_2)\geq \H^{N-1}(\partial B_i)/2$, then
\[
\H^{N-1}\big(\partial B_i\setminus \Gamma_2 \big) \leq \frac{N\omega_N}{2\omega_{N-1}} \,\H^{N-1} (P)
\leq N \H^{N-1} (P)\,.
\]
Putting these two observation together, and assuming without loss of generality that $M\geq N$, one directly gets
\[\begin{split}
P_f(E_i)&\geq P_f(E^*) 
= M \H^{N-1} (\Gamma_1) + \H^{N-1}(\Gamma_2)\\
&\geq M \H^{N-1} (P) + \H^{N-1}(\partial B_i) - \H^{N-1}(\partial B_i \setminus \Gamma_2)
\geq \H^{N-1}(\partial B_i) = P_f(B_i)\,,
\end{split}\]
hence we have concluded in this case (without even using the assumption that $i\in I$).

\case{II}{$\H^{N-1}(\Gamma_2)<\H^{N-1}(\partial B_i)/2$.}
In this case, the Euclidean isoperimetric inequality and the fact that $i\in I$ tell us that
\[
P_\eu (E^*) \geq N\omega_N^{1/N} |E^*|_\eu^{\frac{N-1}N} = N\omega_N^{1/N} \big(|B_i|_\eu-\eps_i\big)^{\frac{N-1}N}
\geq \frac 34\, N\omega_N^{1/N} |B_i|_\eu^{\frac{N-1}N} = \frac 34\, P_f(B_i)\,.
\]
On the other hand, the assumption of this case gives
\[
P_\eu (E) = \H^{N-1}(\Gamma_1)+ \H^{N-1}(\Gamma_2) \leq \H^{N-1}(\Gamma_1) + \frac 12\, P_f(B_i)\,.
\]
The two preceding estimates imply $\H^{N-1}(\Gamma_1) \geq \frac 14\, P_f(B_i)$, which in turn yields
\[
P_f(E) \geq M \H^{N-1}(\Gamma_1) \geq P_f(B_i)\,,
\]
as soon as $M\geq 4$. The proof is then concluded also for this last case.
\end{proof}

The next step is to observe what happens for indices in $J$.
\begin{claim}\label{lastclaim}
One has $\sum_{i\in J} \Big(P_f (E_i ) - P_f(B_i)\Big)\geq - C(N) \eta^{\frac{N-1}N}$, where $C(N)$ is a purely dimensional constant.
\end{claim}
\begin{proof}
The claim is emptily true if $J=\emptyset$, then we can directly suppose $J\neq \emptyset$ and call $j$ the smallest element of $J$. A simple calculation, recalling that $\eps_j > 1/(4\cdot 2^j)$, yields
\[
\sum_{i>j} \eps_i^{\frac{N-1}N}
\leq \sum_{i>j} \frac 1{2^{i\frac{N-1}N}}
= \frac{1}{\big(2^{\frac{N-1}N}-1\big)\cdot 2^{j\,\frac{N-1}N}}
\leq C_1(N) \eps_j^{\frac{N-1}N}\,,
\]
where
\[
C_1(N) = \frac{4^{\frac{N-1}N}}{2^{\frac{N-1}N}-1}\,.
\]
Therefore, using Claim~\ref{cleps}, Claim~\ref{claimI} and~(\ref{reclat}), we deduce
\[\begin{split}
\sum_{i\in J} \Big(P_f (E_i ) - P_f(B_i)\Big) 
&\geq -N\omega_N^{1/N} \bigg( \eps_j^{\frac{N-1}N} + \sum_{i\in J,\,i>j} \eps_i^{\frac{N-1}N}\bigg)
\geq -N\omega_N^{1/N} \bigg( \eps_j^{\frac{N-1}N} + \sum_{i>j} \eps_i^{\frac{N-1}N}\bigg)\\
&\geq -N\omega_N^{1/N}\big( C_1(N) +1\big) \eps_j^{\frac{N-1}N}
= -C(N) \eps_j^{\frac{N-1}N}
\geq -C(N) \Big(\sum\nolimits_{i\in\N} \eps_i\Big)^{\frac{N-1}N}\\
&= -C(N) \eta^{\frac{N-1}N}\,.
\end{split}\]
\end{proof}

We are finally in position to conclude. In fact, putting together~(\ref{tocomp}), (\ref{quickobs}), Claims~\ref{cleps}, \ref{claimI} and~\ref{lastclaim}, and assuming without loss of generality that $E$ has finite perimeter, we have
\[\begin{split}
P_f(E) - P_f(B) &= 
\sum_{i\in\N} P_f (E_i \cup F_i) - P_f(B)\\
&=\sum_{i\in\N} \Big(P_f (E_i \cup F_i) - P_f(B_i\cup F_i)\Big)+\sum_{i\in\N} P_f(B_i\cup F_i) - P_f(B)\\
&\geq \sum_{i\in\N} \Big(P_f (E_i ) - P_f(B_i)\Big)+\sum_{i\in\N} P_f(B_i\cup F_i) - P_f(B)\\
&\geq \sum_{i\in J} \Big(P_f (E_i ) - P_f(B_i)\Big)+\frac{M-1}2\, N \omega_N^{1/N} \eta^{\frac{N-1}N}
\geq \xi \eta^{\frac{N-1}N}\,,
\end{split}\]
where $\xi>0$ if $M$ is big enough. We have then established the validity of~(\ref{toget}), thus it is definitively proved that $B$ is the (unique) isoperimetric set of volume $M$, as desired.

\section{Applications on existence and regularity\label{applications}}

This final section is devoted to show two applications of Theorems~\mref{boundedness} and~\mref{eps-epsbeta} to the questions of the existence and regularity of isoperimetric sets. Even if those two are only simple consequences of the above theorems and of the known facts about existence and regularity, the results that we can find are stronger than those which were previously known. More refined new results concerning the regularity in the $2$-dimensional case are contained in the forthcoming paper~\cite{CP2}.

\subsection{On the existence of isoperimetric sets}

Let us start discussing the question of existence of isoperimetric sets. As we explained in the Introduction, the existence is deeply connected with the boundedness; in particular, some results in~\cite{MP} provide existence of isoperimetric sets (for a certain volume $m$) under the assumption that all the isoperimetric sets for volumes smaller than $m$ (if any) are bounded. Of course, in all these results the boundedness assumption can be removed whenever it comes directly from Theorem~\mref{boundedness}. Let us be more precise: we recall the following result (which can be found in~\cite[Theorems~7.9,\,7.11,\,7.13]{MP}).

\begin{theorem}\label{above}
Let $f$ be a density on $\R^N$ approaching a finite limit $a>0$ at infinity, and assume that the isoperimetric sets are bounded. Then there exist isoperimetric sets of all volumes if one of the following properties holds:
\begin{enumerate}
\item[(i)] for every $V> 0$ and for every $R>0$, there is some ball $B$ of volume $V$ at distance from the origin at least $R$ such that
\[
\sup_{x\in B} f(x) \leq a^{\frac 1N} \Big( \inf_{x\in B} f(x)\Big)^{\frac{N-1} N}\,;
\]
\item[(ii)] $f$ is radial and, for any $c>0$ and any $\rho>0$, there exists some $R\geq \rho$ such that
\[
f(R) \leq a - e^{-cR}\,;
\]
\item[(iii)] for any $V>0$, there exist balls $B$ of volume $V$ arbitrarily far from the origin satisfying the mean inequality
\[
\intmed_{\partial B} f \leq a^{\frac 1N} \bigg( \intmed_B f\bigg)^{\frac{N-1} N}\,.
\]
\end{enumerate}
\end{theorem}

As an immediate application of Theorem~\mref{boundedness} we can then strengthen the above existence result as follows (notice that there is no need of requiring the essential boundedness to $f$, by the assumption that $f$ converges to $a>0$ at infinity).

\begin{theorem}
Let $f$ be a continuous density on $\R^N$ approaching a finite limit $a>0$ at infinity. Then there exist isoperimetric sets of all volumes if one of the properties~{\rm (i)},\,{\rm (ii)}, or~{\rm (iii)} of Theorem~\ref{above} holds true.
\end{theorem}

\subsection{On the regularity of isoperimetric sets}

We pass now to the regularity issue. To start, we recall what is known up to now concerning the regularity of isoperimetric sets in $\R^N$ with density (see for instance~\cite[Proposition~3.5 and Corollary~3.8]{Morgan2003}).

\begin{theorem}\label{strongregu} Let $f$ be a smooth or $C^{k,\alpha}$ density on $\R^n$, with $k\geq 1$. Then the boundary of any isoperimetric set is a smooth or $C^{k+1,\alpha}$ submanifold except on a singular set of Hausdorff dimension at most $n-8$.
\end{theorem}
Notice that, as we already explained in the Introduction, the known result covers only cases in which the density is at least Lipschitz, while there are no results for lower regularity of the density. Our theorems, instead, allow to obtain some regularity results for the isoperimetric sets even with just bounded densities.\par

What we will do here, in fact, is just to put together our Theorem~\mref{eps-epsbeta} and the well-known regularity theory in the standard Euclidean setting. To do so, let us first briefly recall a couple of important notions of minimality for the perimeter and the corresponding regularity properties (whose proofs can be found for instance in~\cite{DavSem,KKLS,Tam}, see also~\cite{GG}); then, we will derive the regularity results for our setting.
\begin{definition}
Let $E\subseteq\R^N$ be a set of locally finite perimeter. We say that $E$ is \emph{quasi-minimal} if, for some $C>0$ and for every ball $B_r(x)$,
\[
P_\eu\big(E,B_r(x)\big)\leq Cr^{N-1}\,.
\]
Moreover, we say that $E$ is \emph{$\omega$-minimal}, for some continuous and increasing function $\omega:\R^+\to\R^+$ with $\omega(0)=0$, if, for every ball $B_r(x)$ and every set $F$ such that $F\triangle E\subset\subset B_r(x)$, one has
\begin{equation}\label{omega-min}
P_\eu\big(E,B_r(x)\big)\leq P_\eu\big(F,B_r(x)\big)+\omega(r) \,r^{N-1}\,.
\end{equation}
\end{definition}
\begin{definition}
Let $E$ be a Borel subset of $\R^N$. We say that $E$ is \emph{porous} if there exists a small constant $\delta>0$ such that, for every $x\in \partial E$ and every arbitrarily small ball $B_r(x)$ centered at $x$, there are two balls $B_1,\,B_2 \subseteq B_r(x)$ of radius $\delta r$ such that $B_1\subseteq E$ and $B_2\subseteq \R^N\setminus E$.
\end{definition}
\begin{theorem}\label{omega-regu}
If a set $E$ is quasi-minimal, then it is porous and the reduced and the topological boundaries coincide ($\H^{N-1}$-a.e.). Moreover, if $E$ is $\omega$-minimal with $\omega(r)=Cr^\eta$ for some $0<\eta\leq 1$, then $\partial^*E$ is $C^{1,\eta/2}$.
\end{theorem}
We will see that an isoperimetric set is always quasi-minimal if the density is even just bounded from above and below, hence the porosity holds true in all these cases. Instead, the $\omega$-minimality holds true as soon as $f$ is H\"older continuous, thus in these more particular cases we will get also some further regularity. More precisely, we obtain the following result.
\begin{theorem}\label{regu}
Let $E$ be an isoperimetric set corresponding to the density $f$. If $f$ is bounded from above and below, then $E$ is porous and $\partial^* E=\partial E$. If moreover $f$ is $\alpha$-H\"older, then one has also that $\partial^*E\in C^{1,\frac{\alpha}{2N(1-\alpha)+2\alpha}}$.
\end{theorem}

We point out that the regularity given by the above theorem is surely not optimal. Indeed, in the forthcoming paper~\cite{CP2} we will improve the above regularity result by showing that, in dimension $N=2$, if $f\in C^{0,\alpha}$ then the boundary of any isoperimetric set is of class $C^{1,\frac{\alpha}{3-2\alpha}}$, while Theorem~\ref{regu} gives only $C^{1,\frac{\alpha}{4-2\alpha}}$.\par

Before giving the proof of Theorem~\ref{regu}, a couple of remarks is in order. 

\begin{remark}
We have claimed the regularity theorem under the assumption that the density is bounded, or H\"older, instead of essentially bounded or essentially H\"older, just for simplicity of notations. However, it is very easy to deduce the general claim. In fact, assume that $f$ is essentially bounded, or essentially $\alpha$-H\"older, and let $U_\delta$ be the open sets as in Definition~\ref{essbdd}. Then, just arguing inside each open set $U_\delta$, in the essentially bounded case we derive that an isoperimetric set $E$ satisfies the porosity property on each $U_\delta$, and that $\partial E=\partial^* E$ on $\bigcup_{\delta>0} U_\delta$. Similarly, in the essentially $\alpha$-H\"older case, we deduce that $\partial^*E\cap \bigcup_{\delta>0} U_\delta$ is of class $C^{1,\frac{\alpha}{2N(1-\alpha)+2\alpha}}$.
\end{remark}

\begin{remark}
One could try to obtain some regularity for the isoperimetric sets also starting from other standard regularity results (whose proofs can be found in~\cite{Amb,Mag,Tam}). For instance, a set $E$ is of class $C^1$ if it has the uniform interior and exterior ball property. Recall that $E$ is said to satisfy the \emph{interior (resp., exterior) ball property} if there is $\bar r>0$ such that, for every $x\in \partial^*E$, there exists a ball $B_{\bar r}(y)$ contained in $E$ (resp., in $\R^N\setminus E$) such that $x\in\partial B_{\bar r}(y)$. In addition, $E$ is even of class $C^{1,\gamma}$ for every $0<\gamma<1$ (and even $C^{1,1}$ if $N=2$) if it is \emph{$\Lambda$-minimal}, that is, there exists $\Lambda\geq 0$ such that for every set $F$ one has $P_\eu(E)\leq P_\eu(F) + \Lambda|F \triangle E|_\eu$.\par
But unfortunately, both these conditions seem to become useful only when $f$ is at least Lipschitz (while we are interested in the lower regularity case). To be more precise, concerning the interior-exterior ball property one can easily observe that, if $f$ is not Lipschitz, it is not even true that an arc of circle of small radius is longer than the corresponding chord. And concerning the $\Lambda$-minimality, one cannot hope to have it if $f$ is $\alpha$-H\"older and $\alpha<1$, as one can derive arguing as in the proof of Theorem~\ref{regu} below, and as one can also guess because otherwise the regularity result ($f$ of class $C^{0,\alpha}$ would imply any isoperimetric set of class $C^{1,1-\eps}$) would be excessively strong.
\end{remark}

We are now ready to conclude the paper with the proof of the regularity Theorem~\ref{regu}.

\begin{proof}[Proof of Theorem~\ref{regu}]
Keeping in mind Theorem~\ref{omega-regu}, we are reduced to check that an isoperimetric set $E$ is quasi-minimal if $f$ is bounded from above and below, while $E$ is $\bar\omega$-minimal with $\bar\omega(r)=\overline C \, r^{\frac{\alpha}{N(1-\alpha)+\alpha}}$ if $f$ is also $\alpha$-H\"older.\par

Let us start by assuming that $1/M < f <M$, and suppose by contradiction that an isoperimetric set $E$ is not quasi-minimal. Hence, for any large constant $K$ there exists a ball $B_r(x)$ such that 
\[
P_\eu\big(E,B_r(x)\big)> K r^{N-1}\,.
\]
Let then $B_r(x)$ be any such ball, and define the set $F:=E\setminus B_{r}(x)$: provided that $K$ is very large, we deduce
\[\begin{split}
P_f(F)&\leq P_f(E)-P_f\big(E, B_r(x)\big)+N\omega_Nr^{N-1}M \\
&\leq P_f(E)-\frac{K}{M}\, r^{N-1}+N\omega_Nr^{N-1}M \leq P_f(E)-\frac{K}{2M}\, r^{N-1}\,.
\end{split}
\]
Thanks to Theorem~\mref{eps-epsbeta}, we know that $F$ fulfills the $\eps-\eps^{\frac{N-1}{N}}$ property with some constant $C$, thus there exists a further set $\widetilde E$ satisfying $|\widetilde E|_f=|E|_f$ and with
\[\begin{split}
P_f(\widetilde E) &\leq P_f(F) + C \Big( \big|\widetilde E\big|_f-\big|F\big|_f\Big)^{\frac{N-1}N}
\leq P_f(E)-\frac{K}{2M}\, r^{N-1} + C \Big( \big|E\big|_f-\big|F\big|_f\Big)^{\frac{N-1}N}\\
&\leq P_f(E)-\frac{K}{2M}\, r^{N-1} + C \Big(M \omega_N r^N\Big)^{\frac{N-1}N}
\leq P_f(E)-\frac{K}{3M}\, r^{N-1} < P_f(E)\,,
\end{split}\]
where we are assuming again $K$ to be large enough. Since this inequality is against the isoperimetric property of $E$, the contradiction shows the quasi-minimality of $E$.\par\medskip

Let us now assume that $f$ is also $\alpha$-H\"older and $E$ is an isoperimetric set: to conclude the proof, we need to show that $E$ is $\bar\omega$-minimal with $\bar\omega(r)=\overline C\, r^{\frac{\alpha}{N(1-\alpha)+\alpha}}$ and some suitable $\overline C$. To do so, we pick $0<\eta\leq 1$ and we investigate whether, for some suitable constant $C$, $E$ is $\omega$-minimal with $\omega (r)=Cr^\eta$ (eventually, we will find that the best choice is $\omega=\bar\omega$).\par

Therefore we suppose that, for any large constant $K$, there exist a ball $B_r(x)$ and a set $F$ with $F \triangle E \subset\subset B_r(x)$, such that
\begin{equation}\label{contr-omega-min}
P_\eu(E,B_r(x))>P_\eu(F,B_r(x))+ K r^{\eta+N-1}\,.
\end{equation}
By the first part of the proof we know that $E$ is quasi-minimal, thus $P_\eu(E, B_r(x))\leq C_1 r^{N-1}$, and by~(\ref{contr-omega-min}) it is then also $P_\eu(F, B_r(x))\leq C_1 r^{N-1}$. Let us assume, just for simplicity of notations, that $\min_{B_r(x)}f=1$; hence, since $f$ is $\alpha$-H\"older, it is $\max_{B_r(x)}f\leq 1+Mr^\alpha$, so that
\begin{align*}
P_f\big(E,B_r(x)\big)\geq P_\eu(E, B_r(x))\,, && P_f\big(F,B_r(x)\big)\leq (1+ M r^\alpha)P_\eu\big(F , B_r(x)\big)\,.
\end{align*}
Recalling~\eqref{contr-omega-min}, we deduce
\begin{equation}\label{omega-1}
\begin{split}
P_f(F)-P_f(E)&=P_f(F,B_r(x))-P_f(E,B_r(x))\\
&\leq (1+M r^\alpha)P_\eu(F , B_r(x)) - P_\eu(E, B_r(x))
\leq C_2 r^{\alpha+N-1}-Kr^{\eta+N-1}\,.
\end{split}
\end{equation}
Applying now Theorem~\mref{eps-epsbeta}, we can define a competitor set $\widetilde E$ with
\begin{align*}
\big| \widetilde E\big|_f = \big| E \big|_f\,, &&
P_f\big(\widetilde E\big)\leq P_f(F) + C_3\Big(\big||F|_f-|\widetilde E|_f\big|\Big)^\beta\leq P_f(F)+C_4 r^{N\beta}\,.
\end{align*}
Combining this estimate with~\eqref{omega-1}, we get
\[
P_f(\widetilde E)-P_f( E)\leq P_f(F)-P_f(E)+C_4r^{N\beta}\leq C_2 r^{\alpha+N-1}-Kr^{\eta+N-1} + C_4 r^{N\beta}\,.
\]
Since by the definition~\eqref{defbeta} of $\beta$ one immediately checks that $\alpha+N-1\geq N\beta$ for any $\alpha \in (0,1]$, we have a contradiction with the optimality of $E$ --with the choice of a sufficiently large $K$-- as soon as $\eta+N-1\leq N\beta$, that is, $\eta\leq \frac\alpha{N(1-\alpha)+\alpha}$. Summarizing, we have shown the $\bar\omega$-minimality of $E$ and hence the proof is concluded.
\end{proof}

\section*{Acknowledgments}

The authors wish to thank Luigi Ambrosio, Guido De Philippis, Nicola Fusco, Francesco Maggi, Frank Morgan and Emanuele Spadaro for some fruitful discussions and for their suggestions on a preliminary version of this paper. Both authors have been supported by the ERC Starting Grant ``AnOptSetCon'' n. 258685, while AP has been also supported by the ERC Advanced Grant ``AnTeGeFI'' n. 226234.

\end{document}